\documentclass[oneside,american,english]{amsart}
\usepackage[T1]{fontenc}
\usepackage[latin9]{inputenc}
\usepackage{babel}
\usepackage{amsthm}
\usepackage{amstext}
\usepackage{amssymb}
\usepackage{esint}
\usepackage[unicode=true,
 bookmarks=true,bookmarksnumbered=false,bookmarksopen=false,
 breaklinks=false,pdfborder={0 0 1},backref=false,colorlinks=false]
 {hyperref}
\hypersetup{pdftitle={On multiple frequency power density measurements II. The full Maxwell's equations},
 pdfauthor={Giovanni S. Alberti}}

\makeatletter
\theoremstyle{plain}
\newtheorem{thm}{\protect\theoremname}
  \theoremstyle{definition}
  \newtheorem{defn}[thm]{\protect\definitionname}
  \theoremstyle{definition}
  \newtheorem{example}[thm]{\protect\examplename}
  \theoremstyle{remark}
  \newtheorem{rem}[thm]{\protect\remarkname}
  \theoremstyle{plain}
  \newtheorem{cor}[thm]{\protect\corollaryname}
  \theoremstyle{plain}
  \newtheorem{prop}[thm]{\protect\propositionname}
  \theoremstyle{plain}
  \newtheorem{lem}[thm]{\protect\lemmaname}

\usepackage{empheq}

\makeatother

  \addto\captionsamerican{\renewcommand{\corollaryname}{Corollary}}
  \addto\captionsamerican{\renewcommand{\definitionname}{Definition}}
  \addto\captionsamerican{\renewcommand{\examplename}{Example}}
  \addto\captionsamerican{\renewcommand{\lemmaname}{Lemma}}
  \addto\captionsamerican{\renewcommand{\propositionname}{Proposition}}
  \addto\captionsamerican{\renewcommand{\remarkname}{Remark}}
  \addto\captionsamerican{\renewcommand{\theoremname}{Theorem}}
  \addto\captionsenglish{\renewcommand{\corollaryname}{Corollary}}
  \addto\captionsenglish{\renewcommand{\definitionname}{Definition}}
  \addto\captionsenglish{\renewcommand{\examplename}{Example}}
  \addto\captionsenglish{\renewcommand{\lemmaname}{Lemma}}
  \addto\captionsenglish{\renewcommand{\propositionname}{Proposition}}
  \addto\captionsenglish{\renewcommand{\remarkname}{Remark}}
  \addto\captionsenglish{\renewcommand{\theoremname}{Theorem}}
  \providecommand{\corollaryname}{Corollary}
  \providecommand{\definitionname}{Definition}
  \providecommand{\examplename}{Example}
  \providecommand{\lemmaname}{Lemma}
  \providecommand{\propositionname}{Proposition}
  \providecommand{\remarkname}{Remark}
\providecommand{\theoremname}{Theorem}

\begin{document}
\selectlanguage{american}%
\global\long\def\R{\mathbb{R}}

\global\long\def\N{\mathbb{N}}

\global\long\def\C{\mathbb{C}}

\global\long\def\Cl{C}

\global\long\def\tr{\mathrm{tr}}

\global\long\def\Se{\mathcal{S}}

\global\long\def\P{L_{\beta_{1},\beta_{2}}^{\infty}(\Omega)}

\global\long\def\V{H_{0}^{1}(\Omega;\C)}

\global\long\def\Vp{H^{-1}(\Omega;\C)}

\global\long\def\Vr{H_{0}^{1}(\Omega)}

\global\long\def\Honer{H^{1}(\Omega)}

\global\long\def\Hone{H^{1}(\Omega;\C)}

\global\long\def\Hhalf{H^{1/2}(\Omega;\C)}

\global\long\def\Hhalfr{H^{1/2}(\Omega)}

\global\long\def\phi{\varphi}

\global\long\def\epsilon{\varepsilon}

\global\long\def\div{{\rm div}}

\global\long\def\ld{L^{2}(\Omega;\C^{3})}

\global\long\def\ldr{L^{2}(\Omega;\R^{3})}

\global\long\def\linf{L^{\infty}(\Omega;\C)}

\global\long\def\linfr{L^{\infty}(\Omega)}

\global\long\def\Conealp{\mathcal{C}^{1,\alpha}(\overline{\Omega};\C)}

\global\long\def\Conealpr{\mathcal{C}^{1,\alpha}(\overline{\Omega})}

\global\long\def\Calpr{\mathcal{C}^{0,\alpha}(\overline{\Omega};\R^{d\times d})}

\global\long\def\Calprsca{\mathcal{C}^{0,\alpha}(\overline{\Omega})}

\global\long\def\Czeroonescal{\mathcal{C}^{0,1}(\overline{\Omega})}

\global\long\def\Calpvect{\mathcal{C}^{0,\alpha}(\overline{\Omega};\R^{d})}

\global\long\def\Cone{\Cl^{1}(\overline{\Omega};\C)}

\global\long\def\Coner{\Cl^{1}(\overline{\Omega})}

\global\long\def\Kad{K_{ad}}

\global\long\def\mina{\lambda}

\global\long\def\maxa{\Lambda}

\selectlanguage{english}%
\global\long\def\Hcurl{H(\curl,\Omega)}

\global\long\def\Hmu{H^{\mu}(\curl,\Omega)}

\global\long\def\Hocurl{H_{0}(\curl,\Omega)}

\global\long\def\Hdiv{H(\div,\Omega)}

\global\long\def\k{\omega}

\global\long\def\div{{\rm div}}

\global\long\def\curl{{\rm curl}}

\global\long\def\supp{{\rm supp}}

\global\long\def\sp{{\rm span}}

\global\long\def\ii{\mathbf{i}}

\global\long\def\E{E_{\k}}

\global\long\def\Ei{E_{\k}^{i}}

\global\long\def\Hi{H_{\k}^{i}}

\global\long\def\H{H_{\k}}

\global\long\def\EE{\tilde{E}_{\k}}

\global\long\def\bo{\partial\Omega}

\global\long\def\Co{\Cl(\overline{\Omega};\C^{3})}

\global\long\def\ve{\theta}

\global\long\def\so{\hat{\sigma}}

\global\long\def\order{\kappa}

\global\long\def\p{s}

\global\long\def\pert{t}

\global\long\def\e{\mathbf{e}}

\title[On multiple frequency power density measurements II]{On multiple frequency power density measurements II.\\
The full Maxwell's equations}

\author{Giovanni S. Alberti}

\address{Mathematical Institute, University of Oxford, Andrew Wiles Building,
Radcliffe Observatory Quarter, Woodstock Road, Oxford, OX2 6GG, United
Kingdom.}

\email{giovanni.alberti@maths.ox.ac.uk}

\date{January 5, 2015}
\begin{abstract}
We shall give conditions on the illuminations $\phi_{i}$ such that
the solutions to Maxwell's equations
\[
\left\{ \begin{array}{l}
\curl E^{i}=\ii\k\mu H^{i}\qquad\text{in \ensuremath{\Omega},}\\
\curl H^{i}=-\ii(\k\epsilon+\ii\sigma)E^{i}\qquad\text{in \ensuremath{\Omega},}\\
E^{i}\times\nu=\phi_{i}\times\nu\qquad\text{on \ensuremath{\partial\Omega},}
\end{array}\right.
\]
satisfy certain non-zero qualitative properties inside the domain
$\Omega$, provided that a finite number of frequencies $\k$ are
chosen in a fixed range. The illuminations are explicitly constructed.
This theory finds applications in several hybrid imaging problems,
where unknown parameters have to be imaged from internal measurements.
Some of these examples are discussed. This paper naturally extends
a previous work of the author {[}Inverse Problems 29 (2013) 115007{]},
where the Helmholtz equation was studied.
\end{abstract}



\maketitle

\section{Introduction}

Let $\Omega\subseteq\R^{3}$ be a smooth bounded domain and consider
the Dirichlet boundary value problem for Maxwell's equations
\begin{equation}
\left\{ \begin{array}{l}
\curl E_{\k}^{\phi}=\ii\k\mu H_{\k}^{\phi}\qquad\text{in \ensuremath{\Omega},}\\
\curl H_{\k}^{\phi}=-\ii(\k\epsilon+\ii\sigma)E_{\k}^{\phi}\qquad\text{in \ensuremath{\Omega},}\\
E_{\k}^{\phi}\times\nu=\phi\times\nu\qquad\text{on \ensuremath{\partial\Omega},}
\end{array}\right.\label{eq:main maxwell}
\end{equation}
where $\mu$, $\epsilon$ and $\sigma$ are uniformly positive definite
real tensors representing the magnetic permeability, the electric
permittivity and the conductivity, respectively, $\phi$ is a given
illumination and $\k>0$. It is well known that the above problem
is well-posed and admits a unique solution $(E_{\k}^{\phi},H_{\k}^{\phi})$ \cite{MONK-2003}.

We want to find suitable illuminations $\phi_{i}$ such that the corresponding
solutions to \eqref{eq:main maxwell} satisfy certain non-zero conditions
in $\Omega$. For example, we may look for three illuminations $\phi_{1}$,
$\phi_{2}$ and $\phi_{3}$ such that
\begin{equation}
\bigl|\det\begin{bmatrix}E_{\k}^{\phi_{1}} & E_{\k}^{\phi_{2}} & E_{\k}^{\phi_{3}}\end{bmatrix}(x)\bigr|>0,\label{eq:det intro}
\end{equation}
or, more generally, for $b$ illuminations $\phi_{1},\dots,\phi_{b}$
such that the corresponding solutions verify $r$ conditions given
by 
\begin{equation}
\left|\zeta_{l}\bigl((E_{\k}^{\phi_{1}},H_{\k}^{\phi_{1}}),\dots,(E_{\k}^{\phi_{b}},H_{\k}^{\phi_{b}})\bigr)(x)\right|>0,\qquad l=1,\dots,r,\label{eq:zeta intro}
\end{equation}
where the maps $\zeta_{l}$ depend on $(E_\omega^{\varphi_i},H_\omega^{\varphi_i})$ and their derivatives.

This problem is generally studied for a fixed frequency $\k>0$. A
possible approach to find suitable illuminations uses complex geometric
optics (CGO) solutions for Maxwell's equations. They were introduced
by Colton and Päivärinta \cite{colton-paivarinta-1992}, and generalize
the CGO solutions originally introduced for the Calderón problem \cite{kavian-2003}
by Sylvester and Uhlmann \cite{sylv1987}. These are highly oscillatory
solutions of \eqref{eq:main maxwell} and can be chosen to satisfy
\eqref{eq:zeta intro} (see \cite{chen-yang-2013}). Unfortunately,
CGO solutions have some drawbacks. First, their construction depends
on the knowledge of the parameters $\mu$, $\epsilon$ and $\sigma$,
which in inverse problems are usually unknown. Second, they may be
numerically unstable to implement due to their high oscillations.
Third, the parameters have to be very smooth. In the elliptic case,
another construction method for the illuminations uses the Runge approximation,
which ensures that locally the solutions of the general problem behave
like the solutions to the constant coefficient case. This approach
was first introduced by Bal and Uhlmann in the case of elliptic equations
\cite{bal2011reconstruction}. As in the case of CGO solutions, the
suitable illuminations are not explicitly constructed and may depend
on the unknown parameters $\mu$, $\epsilon$ and $\sigma$.

In this paper we propose an alternative strategy to this issue based
on the use of multiple frequencies in a fixed range $\Kad=[K_{\text{min}},K_{\text{max}}]$,
for some $0<K_{\text{min}}<K_{\text{max}}$. Namely, given the maps
$\zeta_{l}$, we shall give conditions on the illuminations such that
the corresponding solutions satisfy the required properties, provided
that a finite number of frequencies are used in the range $\Kad$.
These conditions may depend on some (or none) of the parameters. More
precisely, there exist illuminations $\phi_{i}$ and a finite number
of frequencies $K\subseteq\Kad$ such that \eqref{eq:zeta intro}
is satisfied for every $x\in\Omega$ and for some $\k\in K$ depending
on $x$. For example, we shall show that the choice $\{\e_{1},\e_{2},\e_{3}\}$
is sufficient for condition \eqref{eq:det intro}, provided that $\sigma$
is close to a constant. The main idea behind this method is simple:
if the illuminations are suitably chosen then the zero level sets
of functionals depending on $E_{\k}^{\phi}$ and $H_{\k}^{\phi}$
\emph{move} when the frequency changes. The proof is based on the
regularity theory for Maxwell's equations, on the analyticity of the
map $\k\mapsto(E_{\k}^{\phi},H_{\k}^{\phi})$ for complex frequencies
$\k$ and on the fact that the required conditions are satisfied in
$\k=0$.

This work is the natural completion of a previous paper \cite{alberti2013multiple},
where the same problem was studied for the Helmholtz equation. The
generalization to Maxwell's equations is not straightforward. Regularity
theory for Maxwell's equations is less developed than elliptic regularity
theory, and a recent work \cite{ALBERTI-CAPDEBOSCQ-2013} has been
used here to have minimal regularity assumptions on the coefficients.
Well-posedness for \eqref{eq:main maxwell} with $\k\in\C\setminus\Sigma$
follows by a standard argument, but is not as classical as with the
Helmholtz equation, since we consider a complex frequency-dependent
refractive index $\epsilon+\ii\k^{-1}\sigma$.

This theory finds applications in the reconstruction procedures of
several hybrid problems. In hybrid imaging techniques different types
of waves are combined simultaneously to obtain high-resolution and
high-contrast images (see \cite{ammari2011expansion,kuchment2011_review,bal2012_review,alberti-capdeboscq-2014}).
In a first step, internal functionals, or \emph{power densities},
are measured inside the domain. In a second step, the parameters have
to be reconstructed from the knowledge of these power densities. In
most situations, the problem is modeled by the Helmholtz equation
\cite{triki2010,cap2011,bal2011quantitative,ammari2012quantitative,bal2011reconstruction,bal2010inverse,ammari2008elasticity,alessandrini2014global}
or by the full Maxwell's equations \cite{seo-kim-etal-2012,bal2013reconstruction,bal2013hybrid,chen-yang-2013},
and the second step usually requires the availability of solutions
satisfying certain qualitative properties inside the domain, such
as \eqref{eq:det intro} or, more generally, \eqref{eq:zeta intro},
for some maps $\zeta_{l}$ depending on the particular problem under
consideration. These conditions have been shown to be satisfied by
suitable CGO solutions or by means of the Runge approximation in the
elliptic case. The multi-frequency approach described in this work
is an alternative. Since $K$ is finite, the dependence of the frequency
on $x$ does not constitute a source of instability in the reconstruction
(see Section~\ref{sec:Applications-to-Hybrid}).

In \cite{alberti2013multiple} the microwave imaging by ultrasound
deformation technique modeled by the Helmholtz equation was studied as
an application. In this paper we provide three different examples
of applications. Two are related to the problem of reconstructing
the electromagnetic parameters from the knowledge of the magnetic
field inside the domain. The third one deals with an inverse problem
for electro-seismic conversion. In the literature regarding hybrid
inverse problems, the Helmholtz equation approximation, rather than
the full Maxwell's equations, has often been considered. It is likely
that as this domain develops, other reconstruction algorithms requiring
\eqref{eq:zeta intro} to be satisfied for some maps $\zeta_{l}$
will appear, thereby providing possible new applications for this
theory.

The multi-frequency approach described in this work provides a good
alternative to the use of CGO or of the Runge approximation to find
solutions satisfying qualitative properties inside the domain. However,
a few points remain unsolved, and the same aspects are unsolved also
in the case of the Helmholtz equation studied in \cite{alberti2013multiple}.
First, is it possible to find a lower bound in \eqref{eq:zeta intro}?
Second, can we quantify the number of needed frequencies? In the case
of the Helmholtz equation, numerical simulations were performed and
suggest that in two dimensions three frequencies are sufficient. Third,
can we remove the assumption on $\sigma$ being close to a constant
(when \eqref{eq:det intro} is considered)? A possible way to tackle
the last point consists in the use of generic illuminations, that
in some situations may be sufficient for a general $\sigma$ (see
Remark~\ref{rem:dependence}). Some of these issues are addressed in \cite{ALBERTI-b-2014,capalb-analytic}.

This paper is organized as follows. In Section~\ref{sec:Main-Results}
we precisely describe the setting and state the main theoretical result.
Then, we apply this to three particular hybrid problems and study
the reconstruction procedures. In Section~\ref{sec:A-multi-frequency-approach}
we prove the main result, and more details on the multi-frequency
approach are given. In Section~\ref{sec:Applications-to-Hybrid}
the  examples of hybrid problems are discussed in more detail. Section~\ref{sec:The-Maxwell's-Equations}
is devoted to the proof of some results regarding well-posedness,
regularity and analyticity properties of Maxwell's equations which
are stated in Section~\ref{sec:A-multi-frequency-approach}.

\section{\label{sec:Main-Results}Main results}

Let $\Omega\subseteq\R^{3}$ be a $\Cl^{\order+1,1}$ bounded
and simply connected domain for some $\order\in\N^{*}$ with a simply
connected boundary $\partial\Omega$. Let $\mu,\epsilon,\sigma\in L^{\infty}(\Omega;\R^{3\times3})$
be real tensors satisfying the ellipticity conditions
\begin{equation}
\mina\left|\xi\right|^{2}\le\xi\cdot\mu\xi\le\maxa\left|\xi\right|^{2},\quad\mina\left|\xi\right|^{2}\le\xi\cdot\epsilon\xi\le\maxa\left|\xi\right|^{2},\quad\mina\left|\xi\right|^{2}\le\xi\cdot\sigma\xi\le\maxa\left|\xi\right|^{2},\qquad\xi\in\R^{3},\label{eq:assumption_ell}
\end{equation}
for some $\mina,\maxa>0$. Moreover, as far as the regularity of these
coefficients is concerned, we assume that
\begin{equation}
\mu,\epsilon,\sigma\in W^{\order,p}(\Omega;\R^{3\times3})\label{eq:regularity}
\end{equation}
for some $p>3$ and $\order\in\N^{*}$. We shall study problem \eqref{eq:main maxwell}
with illumination satisfying
\begin{equation}
\phi\in W^{\order,p}(\Omega;\C^{3}),\qquad\curl\phi\cdot\nu=0\;\text{on \ensuremath{\partial\Omega}.}\label{eq:assumption phi}
\end{equation}
The second of these conditions is required to make the problem well-posed
if $\k=0$ (see Section~\ref{sec:The-Maxwell's-Equations}). The
natural functional space associated to \eqref{eq:main maxwell} is
\[
\Hcurl=\{u\in L^{2}(\Omega;\C^{3}):\curl u\in\ld\}.
\]

\subsection{A multi-frequency approach to the boundary control of Maxwell's equations}

Let $\Kad=[K_{\text{min}},K_{\text{max}}]$ be the range of frequencies
we have access to, for some $0<K_{\text{min}}<K_{\text{max}}$. 
\begin{defn}
Given a finite set $K\subseteq\Kad$ and illuminations $\phi_{1},\dots,\phi_{b}$
satisfying \eqref{eq:assumption phi}, we say that $K\times\{\phi_{1},\dots,\phi_{b}\}$
is \emph{a set of measurements}.
\end{defn}
Let $K\times\{\phi_{1},\dots,\phi_{b}\}$ be a set of measurements.
For any $(\k,\phi_{i})\in K\times\{\phi_{1},\dots,\phi_{b}\}$ denote
the unique solution to \eqref{eq:main maxwell} in $\Hcurl^{2}$ by
$(\Ei,\Hi)$, namely
\begin{equation}
\left\{ \begin{array}{l}
\curl\Ei=\ii\k\mu\Hi\qquad\text{in \ensuremath{\Omega},}\\
\curl\Hi=-\ii q_{\k}\Ei\qquad\text{in \ensuremath{\Omega},}\\
\Ei\times\nu=\phi_{i}\times\nu\qquad\text{on \ensuremath{\partial\Omega},}
\end{array}\right.\label{eq:combined i}
\end{equation}
where 
\begin{equation}
q_{\k}=\k\epsilon+\ii\sigma.\label{eq:defqom}
\end{equation}
In Proposition~\ref{prop:well-posedness} we shall show that $(\Ei,\Hi)\in\Cl^{\order-1}(\overline{\Omega};\C^{6})$.

The reconstruction procedures in hybrid problems often requires the
availability of a certain number of illuminations such that the corresponding
solutions to \eqref{eq:combined i} and their derivatives up to the
$(\kappa-1)$-th order satisfy some non-zero conditions inside the
domain. Let $b\in\N^{*}$ be the number of illuminations and $r\in\N^{*}$
be the number of non-zero conditions. These conditions depend on the
particular problem under consideration, but it turns out that most
of them can be written in the following form.
\begin{defn}
\label{def:zeta-complete}Let $b,r\in\N^{*}$ be two positive integers,
$\p>0$ and let 
\begin{equation}
\zeta=(\zeta_{1},\dots,\zeta_{r})\colon\Cl^{\order-1}(\overline{\Omega};\C^{6}){}^{b}\longrightarrow\Cl(\overline{\Omega};\C)^{r}\quad\text{be analytic.}\label{eq:definition of zeta}
\end{equation}
A set of measurements $K\times\{\phi_{1},\dots,\phi_{b}\}$ is \emph{$\zeta$-complete}
if for every $x\in\overline{\Omega}$ there exists $\k=\k(x)\in K$
such that\foreignlanguage{american}{
\begin{equation}
\begin{aligned} & \text{1. }\bigl|\zeta_{1}\bigl((E_{\k}^{1},H_{\k}^{1}),\dots,(E_{\k}^{b},H_{\k}^{b})\bigr)(x)\bigr|\ge\p,\\
 & \,\vdots\qquad\qquad\qquad\quad\;\;\vdots\\
 & \text{r. }\bigl|\zeta_{r}\bigl((E_{\k}^{1},H_{\k}^{1}),\dots,(E_{\k}^{b},H_{\k}^{b})\bigr)(x)\bigr|\ge\p.
\end{aligned}
\label{eq:zeta-complete}
\end{equation}
}
\end{defn}
The notion of analyticity for maps between complex Banach spaces will
be given in Section~\ref{sec:A-multi-frequency-approach}. This notion
generalizes the usual one for maps of complex variable, and for now
we anticipate some basic facts (detailed in Lemma~\ref{lem:analytic functions}):
multilinear bounded functions are analytic, the composition of analytic
functions is analytic, and $\zeta$ is analytic if and only if $\zeta_{l}$
is analytic for every $l$.

The analyticity assumption on $\zeta$ is not particularly restrictive,
as it will be clear from the following three examples of maps $\zeta$'s
. In the next subsection, we shall see that these examples are motivated
by some hybrid problems.
\begin{example}
\label{exa:det}Take $b=3$, $r=1$, $\order=1$ and $\zeta^{(1)}$
defined by
\[
\zeta^{(1)}((u_{1},v_{1}),(u_{2},v_{2}),(u_{3},v_{3}))=\det\begin{bmatrix}u_{1} & u_{2} & u_{3}\end{bmatrix},\qquad(u_{i},v_{i})\in\Cl(\overline{\Omega};\C^{6}).
\]
The map $\zeta^{(1)}$ is multilinear and bounded, whence analytic.
In this case, the condition characterizing $\zeta^{(1)}$-complete
sets of measurements is\foreignlanguage{american}{
\begin{equation}
\bigl|\det\begin{bmatrix}E_{\k}^{1} & E_{\k}^{2} & E_{\k}^{3}\end{bmatrix}(x)\bigr|\ge\p.\label{eq:det-complete-1}
\end{equation}
}In other words, \eqref{eq:det-complete-1} signals the availability,
in every point, of three independent electric fields and, in particular,
of one non-vanishing electric field. 
\end{example}

\begin{example}
\label{exa:2}Take $b=6$, $r=1$ and $\order=2$. Consider
the function $\eta\colon\Cl^{1}(\overline{\Omega};\C^{3}){}^{2}\longrightarrow\Cl(\overline{\Omega};\C^{3})$
given by
\begin{equation}
\eta(u_{1},u_{2})=(\nabla u_{1})u_{2}-(\nabla u_{2})u_{1}+\div u_{1}u_{2}-\div u_{2}u_{1}-2\,^{t}(\nabla u_{1})u_{2}+2\,^{t}(\nabla u_{2})u_{1}.\label{eq:eta}
\end{equation}
Define now $\zeta^{(2)}\colon\Cl^{1}(\overline{\Omega};\C^{6}){}^{6}\longrightarrow\Cl(\overline{\Omega};\C)$
by
\[
\zeta^{(2)}((u_{1},v_{1}),\dots,(u_{6},v_{6}))=\det\begin{bmatrix}\eta(u_{1},u_{2}) & \eta(u_{3},u_{4}) & \eta(u_{5},u_{6})\end{bmatrix}.
\]
As before, the map $\zeta^{(2)}$ is multilinear and bounded, whence
analytic. The interpretation of the corresponding condition \eqref{eq:zeta-complete}
is not immediate, and the reader is referred to the following subsection
for an application. 
\end{example}

\begin{example}
\label{exa:3}Take $b=6$, $r=2$ and $\order=2$. As $r=2$, the
map $\zeta^{(3)}$ considered here involves two conditions. The first
one is given by $\zeta^{(2)}$, and the second one refers to the availability
of one non-vanishing electric field. More precisely, we define $\zeta^{(3)}\colon\Cl^{1}(\overline{\Omega};\C^{6}){}^{6}\longrightarrow\Cl(\overline{\Omega};\C)^{2}$
by
\[
\zeta^{(3)}((u_{1},v_{1}),\dots,(u_{6},v_{6}))=\bigl(\zeta^{(2)}((u_{1},v_{1}),\dots,(u_{6},v_{6}))\,,\,(u_{1})_{2}\bigr).
\]
The two components of this map are analytic, and so $\zeta^{(3)}$
is analytic.
\end{example}
Before studying the construction of $\zeta$-complete sets of measurements,
let us make a comment on \eqref{eq:main maxwell}. As already stated
in the introduction, our strategy to make use of several frequencies
starts with a study of \eqref{eq:main maxwell} when $\k=0$. For
$\k>0$, the well-posedness for this problem is classical \cite{MONK-2003}.
However, in order to make the problem well-posed in the case $\k=0$,
we need to add some constraints on $\H$, as it will be shown in Section~\ref{sec:The-Maxwell's-Equations}.
In particular, we look for solutions $(\E,\H)\in\Hcurl\times\Hmu$,
where
\[
\Hmu=\{v\in\Hcurl:\div(\mu v)=0\text{ in }\Omega,\;\mu v\cdot\nu=0\text{ on }\bo\}.
\]
We shall see that problem \eqref{eq:main maxwell} is well-posed in
$\Hcurl\times\Hmu$ for all $\k\ge0$.

The main result of this work deals with the construction of $\zeta$-complete
sets of measurements.
\begin{thm}
\label{thm:zeta-complete}Assume that \eqref{eq:assumption_ell} and
\eqref{eq:regularity} hold. Let $\so\in W^{\order,p}(\Omega;\R^{3\times3})$
be an arbitrary matrix-valued function satisfying \eqref{eq:assumption_ell},
$\phi_{1},\dots,\phi_{b}$ satisfy \eqref{eq:assumption phi} and
$\zeta$ be as in \eqref{eq:definition of zeta}. Suppose that
\begin{equation}
\zeta_{l}\bigl((\hat{E}_{0}^{1},\hat{H}_{0}^{1}),\dots,(\hat{E}_{0}^{b},\hat{H}_{0}^{b})\bigr)(x)\neq0,\qquad x\in\overline{\Omega},\, l=1,\dots,r,\label{eq:assumption k 0}
\end{equation}
where $(\hat{E}{}_{0}^{i},\hat{H}_{0}^{i})\in\Hcurl\times\Hmu$ is
the solution to \eqref{eq:combined i} with $\so$ in lieu of \textup{$\sigma$}
and $\k=0$, namely
\begin{equation}
\left\{ \begin{array}{l}
\curl\hat{E}{}_{0}^{i}=0\qquad\text{in \ensuremath{\Omega},}\\
\div(\so\hat{E}{}_{0}^{i})=0\qquad\text{in \ensuremath{\Omega},}\\
\hat{E}{}_{0}^{i}\times\nu=\phi_{i}\times\nu\qquad\text{on \ensuremath{\partial\Omega},}
\end{array}\right.\qquad\left\{ \begin{array}{l}
\curl\hat{H}_{0}^{i}=\so\hat{E}{}_{0}^{i}\qquad\text{in \ensuremath{\Omega},}\\
\div(\mu\hat{H}_{0}^{i})=0\qquad\text{in \ensuremath{\Omega},}\\
\mu\hat{H}_{0}^{i}\cdot\nu=0\qquad\text{on \ensuremath{\partial\Omega}.}
\end{array}\right.\label{eq: k  0}
\end{equation}
There exists $\delta>0$ such that if $\left\Vert \sigma-\so\right\Vert _{W^{\order,p}(\Omega;\R^{3\times3})}\le\delta$
then we can choose a finite $K\subseteq\Kad$ such that 
\[
K\times\{\phi_{1},\dots,\phi_{b}\}
\]
is a $\zeta$-complete set of measurements.
\end{thm}
In the following remark we discuss assumption \eqref{eq:assumption k 0}
and the dependence of the construction of the illuminations on the
electromagnetic parameters.
\begin{rem}
\label{rem:dependence}Suppose that we are in the simpler case $\so=\sigma$.
This result states that we can construct a $\zeta$-complete set of
measurements, if the illuminations $\phi_{1},\dots,\phi_{b}$ are
chosen in such a way that \eqref{eq:assumption k 0} holds true. It
is in general much easier to satisfy \eqref{eq:assumption k 0} than
\eqref{eq:zeta-complete}, as $\k=0$ makes problem \eqref{eq:combined i}
simpler (see next subsection). Note that \eqref{eq: k  0} does not
depend on $\epsilon$, so that the construction of the illuminations
$\phi_{1},\dots,\phi_{b}$ is independent of $\epsilon$. For the
same reason, their construction may depend on $\sigma$ and $\mu$.
However, in the cases where the maps $\zeta_{l}$ involve only the
electric field $E$, it depends on $\sigma$, and not on $\epsilon$
and $\mu$ (see Corollary~\ref{cor:zeta-complete}).

A typical application of the theorem is in the case where $\sigma$
is a small perturbation of a known constant tensor $\so$. Then, the
construction of the illuminations $\phi_{1},\dots,\phi_{b}$ is independent
of $\sigma$. A similar argument would work if $\mu$ is a small perturbation
of a constant tensor $\hat{\mu}$ . We have decided to omit it for
simplicity, since in the applications we have in mind the maps $\zeta_{l}$
do not depend on the magnetic field $H$, thereby making assumption
\eqref{eq:assumption k 0} independent of $\mu$.

In \cite{alberti2013multiple}, we showed that in the case of the
Helmholtz equation there exist \emph{occulting} illuminations, that
is, boundary conditions for which a finite number of frequencies is
not sufficient, and so the assumption corresponding to \eqref{eq:assumption k 0}
cannot be completely removed. Yet, it is very likely that this assumption
can be considerably weakened. Further investigations in this direction
may lead to a construction totally independent of the electromagnetic
parameters. For example, it is natural to wonder whether generic illuminations
are non-occulting. In the case of the Helmholtz equation, we can prove
that this is the case if we consider the map $\zeta(u)=\nabla u$ \cite{2014albertidphil}.
\end{rem}
In the following remark we discuss the construction of the finite
set of frequencies $K$.
\begin{rem}
\label{rem:frequencies}The proof of the theorem is constructive as
far as the choice of the set $K$ is concerned. Given a converging
sequence in $\Kad$, a suitable finite subsequence will be sufficient.
More details are given in Lemma~\ref{lem:transfer}. An upper bound
for the number of needed frequencies is still unavailable. Numerical
simulations for the Helmholtz model were performed in \cite{alberti2013multiple},
which suggest that three frequencies are sufficient in two dimensions.
We believe that four frequencies should be sufficient in the case
of Maxwell's equations in dimension three.
\end{rem}
Finally, we compare this construction with the CGO approach.
\begin{rem}
\label{rem:regularity of coefficients}The regularity of the coefficients
required for this approach is much lower than the regularity required
if CGO solutions are used. Indeed, if the conditions depend on the
derivatives up to the $(\kappa-1)$-th order, with CGO we need at least $W^{\kappa+2,p}$ \cite{chen-yang-2013}, while with this approach we
require the parameters to be in $W^{\kappa,p}$ for some $p>3$. Note that these regularity assumptions are in some sense minimal: in order for \eqref{eq:zeta-complete} to be meaningful pointwise, we need $E^i_\omega,H^i_\omega\in C^{\kappa-1}$.

Moreover, as discussed in Remark~\ref{rem:dependence} and in $\S$~\ref{sub:Applications-to-Hybrid} below, by using this approach the construction  is usually  explicit and does not require highly oscillatory illuminations.
\end{rem}
In the case where the conditions given by the map $\zeta$ are independent
of the magnetic field $H$, Theorem~\ref{thm:zeta-complete} can
be rewritten in the following form.
\begin{cor}
\label{cor:zeta-complete}Assume that \eqref{eq:assumption_ell} and
\eqref{eq:regularity} hold. Let $\so\in W^{\order,p}(\Omega;\R^{3\times3})$
satisfy \eqref{eq:assumption_ell} and $\zeta$ be as in \eqref{eq:definition of zeta}
and independent of $H$. Take $\psi_{1},\dots,\psi_{b}\in W^{\kappa+1,p}(\Omega;\C)$.
Suppose that
\begin{equation}
\zeta_{l}\bigl(\nabla w^{1},\dots,\nabla w^{b}\bigr)(x)\neq0,\qquad x\in\overline{\Omega},\, l=1,\dots,r,\label{eq:assumption k 0-cor}
\end{equation}
where $w^{i}\in\Hone$ is the solution to
\[
\left\{ \begin{array}{l}
\div(\so\nabla w^{i})=0\qquad\text{in \ensuremath{\Omega},}\\
w^{i}=\psi_{i}\qquad\text{on \ensuremath{\partial\Omega}.}
\end{array}\right.
\]
There exists $\delta>0$ such that if $\left\Vert \sigma-\so\right\Vert _{W^{\order,p}(\Omega;\R^{3\times3})}\le\delta$
then we can choose a finite $K\subseteq\Kad$ such that 
\[
K\times\{\nabla\psi_{1},\dots,\nabla\psi_{b}\}
\]
is a $\zeta$-complete set of measurements.
\end{cor}
In other words, if the required qualitative properties do not depend
on $H$, then the problem of finding $\zeta$-complete sets is completely
reduced to satisfying the same conditions for the gradients of solutions
to the conductivity equation. This last problem has received a considerable
attention in the literature (see e.g. \cite{alessandrinimagnanini1994,cap2009,widlak2012hybrid,bal2011reconstruction,bal-courdurier-2013,bal2012inversediffusion,bal-guo-monard-2013}).

\subsection{\label{sub:Applications-to-Hybrid}Applications to hybrid problems}

In this subsection we apply the theory introduced so far to three
reconstruction procedures arising in hybrid imaging. In each case,
the reconstruction is possible if we have a $\zeta$-complete sets
of measurements, where $\zeta$ is one of the maps discussed in Examples
\ref{exa:det}, \ref{exa:2} and \ref{exa:3}. Only the main points
are discussed here, and full details are given in Section~\ref{sec:Applications-to-Hybrid}.

\subsubsection{\label{subsub:Reconstruction-of-first method}Reconstruction of \texorpdfstring{$\epsilon$}{e}
and \texorpdfstring{$\sigma$}{s} from knowledge of internal magnetic fields - first method}

Combining boundary measurements with an MRI scanner, we can measure
the internal magnetic fields $\Hi$ \cite{seo-kim-etal-2012,bal2013reconstruction}.
Assuming $\mu=1$, the electromagnetic parameters to image are $\epsilon$
and $\sigma$, and both are assumed isotropic. We shall consider two
different reconstruction algorithms. For the first method, the relevant
map $\zeta$ is the determinant $\zeta^{(1)}$ introduced in Example
\ref{exa:det}.

Let $K\times\{\phi_{1},\phi_{2},\phi_{3}\}$ be a set of measurements
and consider problem \eqref{eq:combined i}. We shall show that $q_{\k}$
given by \eqref{eq:defqom} satisfies a first order partial differential
equation in $\Omega$, with coefficients depending on the magnetic
fields, thereby known. This equation is of the form
\[
\nabla q_{\k}M_{\k}^{(1)}=F^{(1)}(\k,q_{\k},\Hi,\Delta\Hi)\qquad\text{in \ensuremath{\Omega},}
\]
where $M_{\k}^{(1)}$ is the $3\times6$ matrix-valued function given
by
\[
M_{\k}^{(1)}=\left[\begin{array}{ccccc}
\curl H_{\k}^{1}\times\e_{1} & \curl H_{\k}^{1}\times\e_{2} & \cdots & \curl H_{\k}^{3}\times\e_{1} & \curl H_{\k}^{3}\times\e_{2}\end{array}\right],
\]
and $F^{(1)}$ is a given vector-valued function. If 
\begin{equation}
\bigl|\det\begin{bmatrix}E_{\k}^{1} & E_{\k}^{2} & E_{\k}^{3}\end{bmatrix}(x)\bigr|>0,\label{eq:det}
\end{equation}
then $M_{\k}^{(1)}(x)$ admits a right inverse $(M_{\k}^{(1)})^{-1}(x)$.
The equation for $q_{\k}$ becomes

\begin{equation}
\nabla q_{\k}(x)=F^{(1)}(\k,q_{\k},\Hi,\Delta\Hi)(M_{\k}^{(1)})^{-1}(x).\label{eq:final-1}
\end{equation}
Suppose now that $K\times\{\phi_{1},\phi_{2},\phi_{3}\}$ is $\det$-complete.
This implies that \eqref{eq:det} is satisfied everywhere in $\Omega$
for some $\k=\k(x)$. We shall see that in this case it is possible
to integrate \eqref{eq:final-1} and reconstruct uniquely $q_{\k}$,
provided that $q_{\k}$ is known at one point of $\overline{\Omega}$.

In this example, we have seen that $\zeta^{(1)}$-complete sets are
sufficient to be able to image the electromagnetic parameters. Moreover,
$\zeta^{(1)}$-complete sets can be explicitly constructed by using
Corollary~\ref{cor:zeta-complete}.
\begin{prop}
\label{prop:det-complete}Assume that \eqref{eq:assumption_ell} and
\eqref{eq:regularity} hold with $\order=1$ and let $\so\in\R^{3\times3}$
be positive definite. There exists $\delta>0$ such that if $\left\Vert \sigma-\so\right\Vert _{W^{1,p}(\Omega;\R^{3\times3})}\le\delta$
then we can choose a finite $K\subseteq\Kad$ such that 
\[
K\times\{\e_{1},\e_{2},\e_{3}\}
\]
is a $\zeta^{(1)}$-complete set of measurements.\end{prop}
\begin{proof}
We want to apply Corollary~\ref{cor:zeta-complete} with $\zeta=\zeta^{(1)}$
and $\psi_{i}=x_{i}$ for $i=1,2,3$. We only need to show that \eqref{eq:assumption k 0-cor}
holds. Since $w^{i}=x_{i}$, for every $x\in\overline{\Omega}$ there
holds
\[
\zeta\bigl(\nabla w^{1},\nabla w^{2},\nabla w^{3}\bigr)(x)=\det\begin{bmatrix}\e_{1} & \e_{2} & \e_{3}\end{bmatrix}=1\neq0,
\]
as desired.
\end{proof}

\subsubsection{\label{subsub:Reconstruction-of-second method}Reconstruction of
\texorpdfstring{$\epsilon$}{e} and \texorpdfstring{$\sigma$}{s} from knowledge of internal magnetic fields
- second method}

We consider the reconstruction procedure discussed in \cite{bal2013reconstruction}
for the same hybrid problem studied before. As we shall see, in this
case the relevant map $\zeta$ is $\zeta^{(2)}$, introduced in Example
\ref{exa:2}.

Let $K\times\{\phi_{1},\dots,\phi_{6}\}$ be a set of measurements
and consider problem \eqref{eq:combined i}. We shall prove that $q_{\k}$
satisfies a first order partial differential equation in $\Omega$,
with coefficients depending on the magnetic fields, thereby known.
This equation is of the form
\begin{equation}
\nabla q_{\k}M_{\k}^{(2)}=q_{\k}F^{(2)}(\curl\Hi)\qquad\text{in \ensuremath{\Omega},}\label{eq:eq with M2}
\end{equation}
where $M_{\k}^{(2)}$ is the $3\times3$ matrix-valued function given
by
\[
M_{\k}^{(2)}=\left[\begin{array}{ccc}
\eta(\curl H_{\k}^{1},\curl H_{\k}^{2}) & \eta(\curl H_{\k}^{3},\curl H_{\k}^{4}) & \eta(\curl H_{\k}^{5},\curl H_{\k}^{6})\end{array}\right],
\]
with $\eta$ given by \eqref{eq:eta}, and $F^{(2)}$ is a vector-valued
function. It turns out that $\eta(\curl H_{\k}^{i},\curl H_{\k}^{j})=-q_{\k}^{2}\eta(\Ei,E_{\k}^{j})$,
whence
\[
M_{\k}^{(2)}=-q_{\k}^{2}\left[\begin{array}{ccc}
\eta(E_{\k}^{1},E_{\k}^{2}) & \eta(E_{\k}^{3},E_{\k}^{4}) & \eta(E_{\k}^{5},E_{\k}^{6})\end{array}\right].
\]
Therefore, if 
\begin{equation}
\bigl|\det\begin{bmatrix}\eta(E_{\k}^{1},E_{\k}^{2}) & \eta(E_{\k}^{3},E_{\k}^{4}) & \eta(E_{\k}^{5},E_{\k}^{6})\end{bmatrix}(x)\bigr|>0,\label{eq:det-1}
\end{equation}
then the equation for $q_{\k}$ becomes

\begin{equation}
\nabla q_{\k}(x)=q_{\k}F^{(2)}(\curl\Hi)(M_{\k}^{(2)})^{-1}(x).\label{eq:final-1-1}
\end{equation}
Suppose now that $K\times\{\phi_{1},\phi_{2},\phi_{3}\}$ is $\zeta^{(2)}$-complete.
This implies that \eqref{eq:det-1} is satisfied everywhere in $\Omega$
for some $\k=\k(x)$. As in the previous case, it is possible to integrate
\eqref{eq:final-1-1} and reconstruct $q_{\k}$ uniquely, provided
that $q_{\k}$ is known at one point of $\overline{\Omega}$.

In this example, we have seen that $\zeta^{(2)}$-complete sets are
sufficient to be able to reconstruct the electromagnetic parameters.
As in the previous case, $\zeta^{(2)}$-complete sets can be explicitly
constructed using Corollary~\ref{cor:zeta-complete}. Let $I$ denote
the $3\times3$ identity matrix.
\begin{prop}
\label{prop:zeta2-complete}Assume that \eqref{eq:assumption_ell}
and \eqref{eq:regularity} hold with $\order=2$ and take $\so>0$.
There exists $\delta>0$ such that if $\left\Vert \sigma-\so I\right\Vert _{W^{2,p}(\Omega;\R^{3\times3})}\le\delta$
then we can choose a finite $K\subseteq\Kad$ such that 
\[
K\times\{\e_{2},\nabla(x_{1}x_{2}),\e_{3},\nabla(x_{2}x_{3}),\e_{1},\nabla(x_{1}x_{3})\}
\]
is a $\zeta^{(2)}$-complete set of measurements.\end{prop}
\begin{proof}
We want to apply Corollary~\ref{cor:zeta-complete} with $\zeta=\zeta^{(2)}$
and $\psi_{1}=x_{2}$, $\psi_{2}=x_{1}x_{2}$, $\psi_{3}=x_{3}$,
$\psi_{4}=x_{2}x_{3}$, $\psi_{5}=x_{1}$ and $\psi_{6}=x_{1}x_{3}$.
We only need to show that \eqref{eq:assumption k 0-cor} holds. Since
$w^{i}=\psi_{i}$, a trivial calculation shows that for every $x\in\overline{\Omega}$
\[
\begin{split}
 \zeta^{(2)}\bigl(\nabla w^{1},\dots,\nabla w^{6}\bigr)(x)&=\det\begin{bmatrix}\eta(\e_{2},\nabla(x_{1}x_{2})) & \eta(\e_{3},\nabla(x_{2}x_{3})) & \eta(\e_{1},\nabla(x_{1}x_{3}))\end{bmatrix} \\ &=1,
\end{split}
\]
as desired.
\end{proof}
In \cite{bal2013reconstruction}, complex geometric optics solutions
are used to make problem \eqref{eq:eq with M2} solvable. With this
approach, the six illuminations are given explicitly and do not depend
on the coefficients. However, it has to be noted that the assumption
$\left\Vert \sigma-\so\right\Vert _{W^{2,p}(\Omega;\R^{3\times3})}\le\delta$
can be restrictive (see Remark~\ref{rem:dependence}).

\subsubsection{\label{subsub:Inverse-problem-of}Inverse problem of electro-seismic
conversion}

Electro-seismic conversion is the generation of a seismic wave in
a fluid-saturated porous material when an electric field is applied
\cite{zhu-haartsen-toksov-2000}. The problem is modeled by the coupling
of Maxwell's equations and Biot's equations. We consider the hybrid
inverse problem introduced in \cite{chen-yang-2013}. In the first
step, by inverting Biot's equation, the quantity $D_{\k}^{i}=L\Ei$
is recovered in $\Omega$, where $L>0$ is a possibly varying coefficient
representing the coupling between electromagnetic and mechanic effects.
In a second step, the electromagnetic parameters $\epsilon$ and $\sigma$
have to be imaged from the knowledge of $D_{\k}^{i}$. As in the previous
cases, we assume that $\mu=1$ and that $\epsilon$ and $\sigma$
are isotropic.

Let $K\times\{\phi_{1},\dots,\phi_{6}\}$ be a set of measurements.
Since $\curl\Hi=-\ii q_{\k}\Ei$, this problem is very similar to
the one discussed in \ref{subsub:Reconstruction-of-second method},
and $L$ plays the role of $-\ii q_{\k}$. Therefore, if $K\times\{\phi_{1},\dots,\phi_{6}\}$
is $\zeta^{(2)}$-complete, it is possible to reconstruct the coefficient
$L$. Once $L$ is known, the electric field can be easily obtained
as $\Ei=L^{-1}D_{\k}^{i}.$ Finally, $\epsilon$ and $\sigma$ can
be reconstructed via
\[
\k q_{\k}E_{\k}^{1}(x)=\curl\curl E_{\k}^{1}(x),
\]
provided that $E_{\k}^{1}$ is non-vanishing. In particular, this
is true if $\left|(E_{\k}^{1})_{2}\right|(x)>0$. This is the second
condition in the definition of $\zeta^{(3)}$ in Example \ref{exa:3}.
Therefore, if $K\times\{\phi_{1},\dots,\phi_{6}\}$ is $\zeta^{(3)}$-complete
then it is possible to uniquely reconstruct $\sigma$ and $\epsilon$
if these are known at one point in $\overline{\Omega}$.

The construction of $\zeta^{(3)}$-complete sets of measurements is
analogous to the construction of $\zeta^{(2)}$-complete sets.
\begin{prop}
\label{prop:zeta3-complete}Assume that \eqref{eq:assumption_ell}
and \eqref{eq:regularity} hold with $\order=2$ and take $\so>0$.
There exists $\delta>0$ such that if $\left\Vert \sigma-\so I\right\Vert _{W^{2,p}(\Omega;\R^{3\times3})}\le\delta$
then we can choose a finite $K\subseteq\Kad$ such that 
\[
K\times\{\e_{2},\nabla(x_{1}x_{2}),\e_{3},\nabla(x_{2}x_{3}),\e_{1},\nabla(x_{1}x_{3})\}
\]
is a $\zeta^{(3)}$-complete set of measurements.
\end{prop}
The same comment given after Proposition~\ref{prop:zeta2-complete}
is relevant here: in \cite{chen-yang-2013}, complex geometric optics
solutions are used to make this problem solvable.

\section{\label{sec:A-multi-frequency-approach}A multi-frequency approach
to the boundary control of Maxwell's equations}

In this section we prove Theorem~\ref{thm:zeta-complete}. In subsection
\ref{sub:Analytic-functions} we study some basic properties of analytic
functions in Banach spaces. In subsection \ref{sub:The-Maxwell's-equations}
we state the main results regarding well-posedness, regularity and
analyticity properties for problem \eqref{eq:main maxwell}. Finally,
in subsection \ref{sub:proof of zeta-complete}, the proof of Theorem~\ref{thm:zeta-complete}
is given.

\subsection{\label{sub:Analytic-functions}Analytic functions}

Analytic functions in a Banach space setting were studied in \cite{taylor37}.
Let $E$ and $E'$ be complex Banach spaces, $D\subseteq E$ be an
open set and take $f\colon D\to E'$. We say that $f$ admits a Gateaux
differential in $x_{0}\in D$ with respect to the direction $y\in E$
if the limit
\[
\lim_{\tau\to0}\frac{f(x_{0}+\tau y)-f(x_{0})}{\tau}
\]
exists in $E'$. We say that $f$ is analytic\emph{ }in $x_{0}$ if
it is continuous in $x_{0}$ and admits a Gateaux differential in
$x_{0}$ with respect to every direction $y\in E$. We say that $f$
is analytic in $D$ (or simply analytic) if it is analytic in every
point of $D$. With this definition, it is clear that this notion
extends the classical notion of analyticity for functions of complex
variable.

The following lemma summarizes some of the basic properties of analytic
functions that are of interest to us.
\begin{lem}
\label{lem:analytic functions}Let $E_{1},\dots,E_{r}$, $E$ and
$E'$ be complex Banach spaces. Let $D\subseteq E$ be an open set.
\begin{enumerate}
\item If $f\colon E_{1}\times\dots\times E_{r}\to E'$ is multilinear and
bounded then $f$ is analytic.
\item If $f\colon D\to E_{1}$ and $g\colon E_{1}\to E'$ are analytic then
$g\circ f\colon D\to E'$ is analytic.
\item Take $f=(f_{1},\dots f_{r})\colon D\to E_{1}\times\dots\times E_{r}$.
Then $f$ is analytic if and only if $f_{l}$ is analytic for every
$l=1,\dots,r$.
\end{enumerate}
\end{lem}
\begin{proof}
Parts (1) and (3) trivially follow from the definition. Part (2) is
shown in \cite{Whittlesey1965}.
\end{proof}

\subsection{\label{sub:The-Maxwell's-equations}Maxwell's equations}

The proofs of the results stated in this subsection are given in Section~\ref{sec:The-Maxwell's-Equations}.

We first study well-posedness for the problem at hand. As it will
be clear from the proof of Theorem~\ref{thm:zeta-complete}, we need
to study problem \eqref{eq:main maxwell} in the more general case
of complex frequency. The case $\k\neq0$ is classical \cite{SOMERSALO-ISAACSON-CHENEY-1992},
and the problem is well-posed except for a discrete set of complex
resonances. Well-posedness in the case $\k=0$ will follow from a
standard argument involving the Helmholtz decomposition.
\begin{prop}
\label{prop:well-posedness}Assume that \eqref{eq:assumption_ell}
and \eqref{eq:assumption phi} hold for some $p>3$ and $\order=1$.
There exists a discrete set $\Sigma\subseteq\C\setminus\{0\}$ such
that for all $\k\in\C\setminus\Sigma$ the problem
\begin{equation}
\left\{ \begin{array}{l}
\curl\E=\ii\k\mu\H\qquad\text{in \ensuremath{\Omega},}\\
\curl\H=-\ii q_{\k}\E\qquad\text{in \ensuremath{\Omega},}\\
\div(\mu\H)=0\qquad\text{in \ensuremath{\Omega},}\\
\E\times\nu=\phi\times\nu\qquad\text{on \ensuremath{\partial\Omega},}\\
\mu\H\cdot\nu=0\qquad\text{on \ensuremath{\partial\Omega}.}
\end{array}\right.\label{eq:combined}
\end{equation}
admits a unique solution $(\E,\H)\in\Hcurl\times\Hmu$ and
\[
\left\Vert (\E,\H)\right\Vert _{\Hcurl^{2}}\le C\left\Vert \phi\right\Vert _{\Hcurl},
\]
for some $C>0$ independent of $\phi$. 
\end{prop}
We next state a regularity result for the solution $(\E,\H)$, which
follows from the regularity theorems in \cite{ALBERTI-CAPDEBOSCQ-2013}.
\begin{prop}
\label{prop:regularity}Assume that \eqref{eq:assumption_ell}, \eqref{eq:regularity}
and \eqref{eq:assumption phi} hold for some $p>3$ and $\order\in\N^{*}$.
For $\k\in\C\setminus\Sigma$ let $(\E,\H)\in\Hcurl$ be the unique
solution to \eqref{eq:combined}. Then $(\E,\H)\in\Cl^{\order-1}(\overline{\Omega};\C^{6})$
and
\[
\left\Vert (\E,\H)\right\Vert _{\Cl^{\order-1}(\overline{\Omega};\C^{6})}\le C\left\Vert \phi\right\Vert _{W^{\order,p}(\Omega;\C^{3})},
\]
for some $C>0$ independent of $\phi$. Moreover, if $\k=0$ there
holds 
\[
\left\Vert (E_{0},H_{0})\right\Vert _{W^{\order,p}(\Omega;\C^{6})}\le C'\left\Vert \phi\right\Vert _{W^{\order,p}(\Omega;\C^{3})},
\]
for some $C'>0$ depending on $\Omega$, $\mina$, $\maxa$, $\kappa$,
$p$ and $\left\Vert (\mu,\epsilon,\sigma)\right\Vert _{W^{\order,p}(\Omega;\R^{3\times3})^{3}}$
only.
\end{prop}
We finally state that $(\E,\H)$ depends analytically on $\k$. This
is the main result of Section~\ref{sec:The-Maxwell's-Equations},
and represents one of the main ingredients of the proof of Theorem~\ref{thm:zeta-complete}, 
\begin{prop}
\label{prop:analyticity global}Under the hypotheses of Proposition~\ref{prop:regularity}
the map
\[
S_{\phi}\colon\C\setminus\Sigma\longrightarrow\Cl^{\order-1}(\overline{\Omega};\C^{6}),\qquad\k\longmapsto(\E,\H)
\]
is analytic.
\end{prop}

\subsection{\label{sub:proof of zeta-complete}Proof of Theorem~\ref{thm:zeta-complete}}

The rest of this section is devoted to the proof of Theorem~\ref{thm:zeta-complete}.
We first need two preliminary lemmata. 

One of the main tools of the multi-frequency approach is the study
of \eqref{eq:zeta-complete} in the case $\k=0$. The following lemma
shows that assumption \eqref{eq:assumption k 0} implies that \eqref{eq:zeta-complete}
is satisfied in $\k=0$ if the conductivity $\sigma$ is a perturbation
of $\so$, namely
\[
\sigma=\so+\pert,\qquad\left\Vert \pert\right\Vert _{W^{\order,p}(\Omega;\R^{3\times3})}\le\delta,
\]
for some sufficiently small $\delta>0$. Denote the solution to \eqref{eq:combined i}
corresponding to $\sigma=\so+\pert$ and $\k=0$ by $(E_{0}^{i}(\pert),H_{0}^{i}(\pert))$.
\begin{lem}
\label{lem:perturbation}Assume that the hypotheses of Theorem~\ref{thm:zeta-complete}
hold true and write $\sigma=\so+\pert$. There exists $\delta>0$
such that if $\left\Vert \pert\right\Vert _{W^{\order,p}(\Omega;\R^{3\times3})}\le\delta$
then\foreignlanguage{american}{
\begin{equation}
\bigl|\zeta_{l}\bigl((E_{0}^{1}(\pert),H_{0}^{1}(\pert)),\dots,(E_{0}^{b}(\pert),H_{0}^{b}(\pert))\bigr)(x)\bigr|>0,\qquad l=1,\dots,r,\, x\in\overline{\Omega}.\label{eq:intermediate}
\end{equation}
}\end{lem}
\begin{proof}
Assume $\left\Vert \pert\right\Vert _{W^{\kappa,p}(\Omega;\R^{3\times3})}\le\eta$,
for some $\eta>0$ sufficiently small so that $\so+\pert$ always
satisfies the ellipticity condition \eqref{eq:assumption_ell}. We
use the decomposition $E_{0}^{i}(\pert)=\tilde{E}_{0}^{i}(\pert)+\phi_{i}$,
where
\[
\tilde{E}_{0}^{i}(\pert)\in\Hocurl=\{u\in\Hcurl:u\times\nu=0\text{ on }\partial\Omega\}.
\]
By construction, $(\tilde{E}_{0}^{i}(\pert),H_{0}^{i}(\pert))$ is
a solution to
\[
\left\{ \begin{array}{l}
\curl H_{0}^{i}(\pert)-(\so+\pert)\tilde{E}_{0}^{i}(\pert)=(\so+\pert)\phi_{i}\qquad\text{in \ensuremath{\Omega},}\\
\curl\tilde{E}_{0}^{i}(\pert)=-\curl\phi\qquad\text{in \ensuremath{\Omega},}
\end{array}\right.
\]
together with $\div(\mu H_{0}^{i}(\pert))=0$ in $\Omega$ and $\mu H_{0}^{i}(\pert)\cdot\nu=0$
on $\bo$. Writing $u^{i}(\pert)=E_{0}^{i}(\pert)-E_{0}^{i}(0)=\tilde{E}_{0}^{i}(\pert)-\tilde{E}_{0}^{i}(0)$
and $v^{i}(\pert)=H_{0}^{i}(\pert)-H_{0}^{i}(0)$, an easy calculation
shows that
\[
\left\{ \begin{array}{l}
\curl v^{i}(\pert)-\so u^{i}(\pert)=\pert E_{0}^{i}(\pert)\qquad\text{in \ensuremath{\Omega},}\\
\curl u^{i}(\pert)=0\qquad\text{in \ensuremath{\Omega}.}
\end{array}\right.
\]
We can now apply Proposition~\ref{pro:regularity lemma} (see Section~\ref{sec:The-Maxwell's-Equations})
to obtain that
\[
\left\Vert (u^{i}(\pert),v^{i}(\pert))\right\Vert _{W^{\order,p}(\Omega;\C^{3})^{2}}\le c_{1}\bigl\Vert\pert E_{0}^{i}(\pert)\bigr\Vert_{W^{\order,p}(\Omega;\C^{3})},
\]
for some $c_{1}>0$ independent of $\pert$ (but depending on $\eta$).
Moreover, by Proposition~\ref{prop:regularity} we have $\bigl\Vert E_{0}^{i}(\pert)\bigr\Vert_{W^{\order,p}(\Omega;\C^{3})}\le c_{2}$
for some $c_{2}$ independent of $\pert$. Thus, the Sobolev Embedding
Theorem yields
\begin{equation}
\left\Vert \left(E_{0}^{i}(\pert)-E_{0}^{i}(0),H_{0}^{i}(\pert)-H_{0}^{i}(0)\right)\right\Vert _{\Cl^{\order-1}(\overline{\Omega};\C^{6})}\le c\left\Vert \pert\right\Vert _{W^{\order,p}(\Omega;\R^{3\times3})},\label{eq:bound perturbation constant}
\end{equation}
for some $c>0$ independent of $\pert$.

Note now that \eqref{eq:assumption k 0} yields\foreignlanguage{american}{
\[
\left|\zeta_{l}\bigl((E_{0}^{1}(0),H_{0}^{1}(0)),\dots,(E_{0}^{b}(0),H_{0}^{b}(0))\bigr)(x)\right|>0,\qquad l=1,\dots,r,\, x\in\overline{\Omega}.
\]
}Therefore, in view of \eqref{eq:bound perturbation constant} and
the continuity of $\zeta$ we obtain the result.
\end{proof}
In order to prove Theorem~\ref{thm:zeta-complete}, we still need
to \emph{transfer} property \eqref{eq:intermediate} to any range
of frequencies. This is the content of Lemma~\ref{lem:transfer},
that generalizes \cite[Lemma 4.1]{alberti2013multiple} to Maxwell's
equations.
\begin{lem}
\label{lem:transfer}Assume that \eqref{eq:assumption_ell} and \eqref{eq:regularity}
hold. Let $\phi_{1},\dots,\phi_{b}$ satisfy \eqref{eq:assumption_ell}
and $\zeta$ be as in \eqref{eq:definition of zeta}. Assume that
for $\k=0$
\[
\zeta_{l}\bigl((E_{0}^{1},H_{0}^{1}),\dots,(E_{0}^{b},H_{0}^{b})\bigr)(x)\neq0,\qquad x\in\overline{\Omega},\, l=1,\dots,r.
\]
Take $\k_{n},\k\in\Kad$ with $\k_{n}\to\k$ and $\k_{n}\neq\k$.
Then there exists a finite $N\subseteq\N$ such that
\[
\sum_{n\in N}\left|\zeta_{l}\bigl((E_{\k_{n}}^{1},H_{\k_{n}}^{1}),\dots,(E_{\k_{n}}^{b},H_{\k_{n}}^{b})\bigr)(x)\right|>0,\qquad x\in\overline{\Omega},\, l=1,\dots,r.
\]
In particular, we can choose a finite $K\subseteq\Kad$ such that
\[
\sum_{\k\in K}\left|\zeta_{l}\bigl((E_{\k}^{1},H_{\k}^{1}),\dots,(E_{\k}^{b},H_{\k}^{b})\bigr)(x)\right|>0,\qquad x\in\overline{\Omega},\, l=1,\dots,r.
\]
\end{lem}
\begin{proof}
Take $x\in\overline{\Omega}.$ Consider the map $g_{x}\colon\C\setminus\Sigma\to\C$
defined by 
\[
g_{x}(\k)=\prod_{l=1}^{r}\zeta_{l}\bigl((E_{\k}^{1},H_{\k}^{1}),\dots,(E_{\k}^{b},H_{\k}^{b})\bigr)(x),\qquad\k\in\C\setminus\Sigma.
\]
Proposition~\ref{prop:analyticity global} states that the map $\k\in\C\setminus\Sigma\mapsto(\Ei,\Hi)\in\Cl^{\order-1}(\overline{\Omega},\C^{6})$
is analytic. Hence, in view of \eqref{eq:definition of zeta}, Lemma~\ref{lem:analytic functions}
yields that $g_{x}$ is analytic. Thus, as $g_{x}(0)\neq0$, the set
$\left\{ \k'\in\C\setminus\Sigma:g_{x}(\k')=0\right\} $ has no accumulation
points in $\C\setminus\Sigma$ by the analytic continuation theorem.
Since by assumption $\k$ is an accumulation point for the sequence
$(\k_{n})$, this implies that there exists $n_{x}\in\N$ such that
$g_{x}(\k_{n_{x}})\neq0$. As a consequence, in view of the continuity
of the map $\prod_{l=1}^{r}\zeta_{l}\bigl((E_{\k_{n_{x}}}^{1},H_{\k_{n_{x}}}^{1}),\dots,(E_{\k_{n_{x}}}^{b},H_{\k_{n_{x}}}^{b})\bigr)$,
we can find $r_{x}>0$ such that 
\begin{equation}
g_{y}(\k_{n_{x}})\neq0,\qquad y\in B(x,r_{x})\cap\overline{\Omega}.\label{eq:10}
\end{equation}
Since $\overline{\Omega}=\bigcup_{x\in\overline{\Omega}}\left(B(x,r_{x})\cap\overline{\Omega}\right)$,
there exist $x_{1},\dots,x_{M}\in\overline{\Omega}$ satisfying 
\begin{equation}
\overline{\Omega}=\bigcup_{m=1}^{M}\left(B(x_{m},r_{x_{m}})\cap\overline{\Omega}\right).\label{eq:11}
\end{equation}
Defining $N=\left\{ n_{x_{m}}:m=1:,\dots,M\right\} ,$ by \eqref{eq:10}
and \eqref{eq:11} we obtain the result.
\end{proof}
We are now in a position to prove Theorem~\ref{thm:zeta-complete}
\begin{proof}[Proof of Theorem \ref{thm:zeta-complete}]
Take $\delta$ as in Lemma~\ref{lem:perturbation} and suppose $ $$\left\Vert \sigma-\so\right\Vert _{W^{\order,p}(\Omega;\R^{3\times3})}\le\delta$.
By Lemma~\ref{lem:perturbation} we have
\[
\zeta_{l}\bigl((E_{0}^{1},H_{0}^{1}),\dots,(E_{0}^{b},H_{0}^{b})\bigr)(x)\neq0,\qquad l=1,\dots,r,\, x\in\overline{\Omega}.
\]
As a result, in view of Lemma~\ref{lem:transfer}, we can choose
a finite $K\subseteq\Kad$ such that
\[
\sum_{\k\in K}\left|\zeta_{l}\bigl((E_{\k}^{1},H_{\k}^{1}),\dots,(E_{\k}^{b},H_{\k}^{b})\bigr)(x)\right|>0,\qquad x\in\overline{\Omega},\, l=1,\dots,r.
\]
Finally, as $\overline{\Omega}$ is compact and $\zeta_{l}\bigl((E_{\k}^{1},H_{\k}^{1}),\dots,(E_{\k}^{b},H_{\k}^{b})\bigr)$
are continuous maps, we obtain that $K\times\{\phi_{1},\dots,\phi_{b}\}$
is a $\zeta$-complete set of measurements for some $\p>0$.
\end{proof}

\section{\label{sec:Applications-to-Hybrid}Applications to hybrid problems}

In this section we analyze the hybrid problems introduced in Subsection~\ref{sub:Applications-to-Hybrid}
in more detail. The inverse problem for electro-seismic conversion
was fully studied in \ref{subsub:Inverse-problem-of} and no further
details will be given here.

\subsection{\label{sub:Reconstruction-of-first method}Reconstruction of \texorpdfstring{$\epsilon$}{e}
and \texorpdfstring{$\sigma$}{s} from knowledge of internal magnetic fields - first method}

Let us recall the problem discussed in \ref{subsub:Reconstruction-of-first method}.
Combining boundary measurements with an MRI scanner, we are able to
measure the magnetic fields. From the knowledge of $H$, the electromagnetic
parameters have to be imaged. Consider problem \eqref{eq:combined i}
\begin{equation}
\left\{ \begin{array}{l}
\curl\Ei=\ii\k\Hi\qquad\text{in \ensuremath{\Omega},}\\
\curl\Hi=-\ii q_{\k}\Ei\qquad\text{in \ensuremath{\Omega},}\\
\Ei\times\nu=\phi_{i}\times\nu\qquad\text{on \ensuremath{\partial\Omega}.}
\end{array}\right.\label{eq:model in applications}
\end{equation}
Recall that $q_{\k}=\k\epsilon+\ii\sigma$. We assume $\mu=1$ and
$\epsilon,\sigma>0$, namely we study the isotropic case. Given a
set of measurements $K\times\{\phi_{1},\dots,\phi_{b}\}$, we measure
$H_{\k}^{i}$ in $\Omega$ and want to reconstruct $\epsilon$ and
$\sigma$.

Two interesting issues of practical importance are not considered
in this work. First, the case where one or two components of the magnetic
fields are measured. In such a case, the rotation of the MRI scanner
is avoided. The reader is referred to \cite{seo2003reconstruction},
where a low frequency approximation is considered. Second, it is possible
to consider anisotropic coefficients, which in some cases are a better
model for human tissues \cite{bal2013reconstruction}.

We now describe the first method to reconstruct $\sigma$ and $\epsilon$.
Let $K\times\{\phi_{1},\phi_{2},\phi_{3}\}$ be a $\zeta^{(1)}$-complete
set of measurement (see Proposition~\ref{prop:det-complete}). Namely,
for every $x\in\overline{\Omega}$ there exists $\k(x)\in K$ such
that 
\[
\bigl|\det\bigl[\begin{array}{ccc}
E_{\k(x)}^{1} & E_{\k(x)}^{2} & E_{\k(x)}^{3}\end{array}\bigr](x)\bigr|\ge\p',
\]
for some $\p'>0$ independent of $x$. Thus, \eqref{eq:model in applications}
implies
\begin{multline}
\bigl|\det\bigl[\begin{array}{ccc}
\curl H_{\k(x)}^{1} & \curl H_{\k(x)}^{2} & \curl H_{\k(x)}^{3}\end{array}\bigr](x)\bigr|\\=\bigl|q_{\k(x)}^{3}\det\bigl[\begin{array}{ccc}
E_{\k(x)}^{1} & E_{\k(x)}^{2} & E_{\k(x)}^{3}\end{array}\bigr](x)\bigr|\ge\p,\label{eq:det condition}
\end{multline}
for some $\p>0$ independent of $x$. This inequality will be necessary
in the following.

We now proceed to eliminate the unknown electric field from system
\eqref{eq:model in applications}, in order to obtain an equation
with only $\epsilon$ and $\sigma$ as unknowns and the magnetic field
as a known datum. An immediate calculation shows that for any $\k\in K$
and $i=1,2,3$ there holds $\curl(\ii q_{\k}^{-1}\curl\Hi)=\ii\k\Hi$
in $\Omega$, whence
\[
\nabla q_{\k}\times\curl\Hi=q_{\k}\curl\curl\Hi-q_{\k}^{2}\k\Hi=-q_{\k}\Delta\Hi-q_{\k}^{2}\k\Hi\qquad\text{in \ensuremath{\Omega},}
\]
where the last identity is a consequence of the fact that $\Hi$ is
divergence free, since $\mu=1$. Taking now scalar product with $\e_{j}$
for $j=1,2$ we have
\[
\nabla q_{\k}\cdot\bigl(\curl\Hi\times\e_{j}\bigr)=-q_{\k}\Delta(\Hi)_{j}-q_{\k}^{2}\k(\Hi)_{j}\qquad\text{in \ensuremath{\Omega}.}
\]
We can now write these 6 equations in a more compact form. By introducing
the $3\times6$ matrix
\[
M_{\k}^{(1)}=\left[\begin{array}{ccccc}
\curl H_{\k}^{1}\times\e_{1} & \curl H_{\k}^{1}\times\e_{2} & \cdots & \curl H_{\k}^{3}\times\e_{1} & \curl H_{\k}^{3}\times\e_{2}\end{array}\right]
\]
and the six-dimensional horizontal vector
\[
v_{\k}=\left((H_{\k}^{1})_{1},(H_{\k}^{1})_{2},\dots,(H_{\k}^{3})_{1},(H_{\k}^{3})_{2}\right)
\]
we obtain
\begin{equation}
\nabla q_{\k}M_{\k}^{(1)}=-q_{\k}\Delta v_{\k}-q_{\k}^{2}\k v_{\k}\qquad\text{in \ensuremath{\Omega}.}\label{eq:almost final}
\end{equation}
We now want to right invert the matrix $M_{\k}^{(1)}$ to obtain a
well-posed first order PDE for $q_{\k}$. The following lemma gives
a sufficient condition for the matrix $M_{\k}^{(1)}$ to admit a right
inverse.
\begin{lem}
\label{lem:full rank}Let $G_{1},G_{2},G_{3}\in\C^{3}$ be linearly
independent. Then the $3\times6$ matrix
\[
\left[\begin{array}{ccccc}
G_{1}\times\e_{1} & G_{1}\times\e_{2} & \cdots & G_{3}\times\e_{1} & G_{3}\times\e_{2}\end{array}\right]
\]
has rank three.\end{lem}
\begin{proof}
Take $u\in\C^{3}$ such that $G_{i}\times\e_{j}\cdot u=0$ for every
$i=1,2,3$ and $j=1,2$. We need to prove that $u=0$. Since $\e_{j}\times u\cdot G_{i}=0$
for all $i$ and $j$ and $\{G_{i}:i=1,2,3\}$ is a basis of $\C^{3}$
by assumption, we obtain $\e_{j}\times u=0$ for $j=1,2$, namely
$u=0$.
\end{proof}
Define now for any $\k\in K$ the set
\[
\Omega_{\k}=\{x\in\overline{\Omega}:\bigl|\det\begin{bmatrix}\curl H_{\k}^{1} & \curl H_{\k}^{2} & \curl H_{\k}^{3}\end{bmatrix}(x)\bigr|>\frac{\p}{2}\}.
\]
Since $\frac{\p}{2}<\p$, in view of \eqref{eq:det condition} we
obtain the cover
\[
\overline{\Omega}=\bigcup_{\k\in K}\Omega_{\k}.
\]
As the sets $\Omega_{\k}$ are relatively open in $\overline{\Omega}$,
they must overlap, and this will be exploited below in the reconstruction.
For any $\k\in K$ and $x\in\Omega_{\k}$, in view of Lemma~\ref{lem:full rank}
the matrix $M_{\k}^{(1)}(x)$ admits a right inverse, which with an
abuse of notation we denote by $(M_{\k}^{(1)})^{-1}(x)$. Therefore,
problem \eqref{eq:almost final} becomes
\begin{equation}
\nabla q_{\k}=-q_{\k}\Delta v_{\k}(M_{\k}^{(1)})^{-1}-q_{\k}^{2}\k v_{\k}(M_{\k}^{(1)})^{-1}\qquad\text{in }\Omega_{\k}.\label{eq:final}
\end{equation}

It is now possible to integrate this PDE and reconstruct $\epsilon$
and $\sigma$ in every $x\in\Omega$ if these are known for one value
$x_{0}\in\overline{\Omega}$. It is the nature of the multi-frequency
approach that the relevant conditions are satisfied only locally for
a fixed value of the frequency $\k\in K$. In other words, \eqref{eq:final}
is not satisfied everywhere but only in $\Omega_{\k}$. As a consequence,
it is not possible to reconstruct $\epsilon$ and $\sigma$ in $x$
after one simple integration of \eqref{eq:final}. The process is
more involved, and is similar to the algorithm described in \cite{bal2012inversediffusion}.

Suppose now that $q_{\k}(x_{0})$ is known for some $x_{0}\in\overline{\Omega}$
and take $x\in\overline{\Omega}$. Let $\Omega_{\k}^{j}$ for $j\in J_{\k}$
be the connected components of $\Omega_{\k}$, for a suitable set
$J_{\k}$. Since $\overline{\Omega}$ is compact, from the cover $\overline{\Omega}=\bigcup_{\k\in K}\bigcup_{j\in J_{\k}}\Omega_{\k}^{j}$,
we can extract a finite subcover
\[
\overline{\Omega}=\bigcup_{\k\in K}\bigcup_{j\in J'_{\k}}\Omega_{\k}^{j},
\]
where $J'_{\k}\subseteq J_{\k}$ is finite. Hence, as $\overline{\Omega}$
is connected and $\Omega_{\k}^{j}$ are relatively open in $\overline{\Omega}$
and connected, we can find a smooth path $\gamma\colon[0,1]\to\overline{\Omega}$
such that $\gamma(0)=x_{0}$, $\gamma(1)=x$ and
\[
\gamma([0,1])=\bigcup_{m=0}^{M-1}\gamma([t_{m},t_{m+1}]),
\]
for some $M\in\N^{*}$, where $t_{0}=0$, $t_{M}=1$ and $\gamma([t_{m},t_{m+1}])\subseteq\Omega_{\k_{m}}^{j_{m}}$
for some $\k_{m}\in K$ and $j_{m}\in J'_{\k_{m}}$. Starting with
$m=0$, we integrate \eqref{eq:final} with $\k=\k_{m}$ along $\gamma([t_{m},t_{m+1}])$
and obtain $q_{\k_{m}}$ in $\gamma(t_{m+1})$. Thus, we can reconstruct
$\sigma$ and $\epsilon$ in $\gamma(t_{m+1})$ and so $q_{\k_{m+1}}(\gamma(t_{m+1}))$.
Repeating this process $M-1$ times we obtain $\epsilon(x)$ and $\sigma(x)$,
as desired.

\subsection{\label{sub:Reconstruction-of-second method}Reconstruction of \texorpdfstring{$\epsilon$}{e}
and \texorpdfstring{$\sigma$}{s} from knowledge of internal magnetic fields - second
method}

We consider here the problem of Subsection~\ref{sub:Reconstruction-of-first method}
and give details of the reconstruction algorithm summarized in~\ref{subsub:Reconstruction-of-second method}.

Let $K\times\{\phi_{1},\dots,\phi_{6}\}$ be a $\zeta^{(2)}$-complete
set of measurement (see Proposition~\ref{prop:zeta2-complete}).
Namely, for every $x\in\overline{\Omega}$ there exists $\k(x)\in K$
such that 
\begin{equation}
\bigl|\det\bigl[\begin{array}{ccc}
\eta(E_{\k(x)}^{1},E_{\k(x)}^{2}) & \eta(E_{\k(x)}^{3},E_{\k(x)}^{4}) & \eta(E_{\k(x)}^{5},E_{\k(x)}^{6})\end{array}\bigr](x)\bigr|\ge\p,\label{eq:second method-0}
\end{equation}
for some $\p>0$ independent of $x$, where $\eta$ is given by \eqref{eq:eta}. 

In view of \eqref{eq:model in applications}, for $\k\in K$ there
holds
\[
\curl\curl E_{\k}^{1}\cdot E_{\k}^{2}-\curl\curl E_{\k}^{2}\cdot E_{\k}^{1}=0\qquad\text{in }\Omega.
\]
An easy calculation shows that substituting $\Ei=\ii q_{\k}^{-1}\curl\Hi$
in this identity we obtain
\begin{equation}
\nabla q_{\k}\cdot\eta(\curl H_{\k}^{1},\curl H_{\k}^{2})=q_{\k}\,\gamma(\curl H_{\k}^{1},\curl H_{\k}^{2})\qquad\text{in }\Omega,\label{eq:second method}
\end{equation}
where $\gamma\colon\Cl^{1}(\overline{\Omega};\C^{3})^{2}\to\Cl(\overline{\Omega};\C)$
is defined by
\[
\gamma(u_{1},u_{2})=(\nabla\div u_{1})\cdot u_{2}-(\nabla\div u_{2})\cdot u_{1}-\Delta u_{1}\cdot u_{2}+\Delta u_{2}\cdot u_{1}.
\]
Repeating the same argument with the other illuminations and combining
the resulting equations we have
\begin{equation}
\nabla q_{\k}M_{\k}^{(2)}=q_{\k}(\gamma(\curl H_{\k}^{1},\curl H_{\k}^{2}),\gamma(\curl H_{\k}^{3},\curl H_{\k}^{4}),\gamma(\curl H_{\k}^{5},\curl H_{\k}^{6}))\qquad\text{in \ensuremath{\Omega},}\label{eq:second method-1}
\end{equation}
where $M_{\k}^{(2)}$ is the $3\times3$ matrix-valued function given
by
\[
M_{\k}^{(2)}=\left[\begin{array}{ccc}
\eta(\curl H_{\k}^{1},\curl H_{\k}^{2}) & \eta(\curl H_{\k}^{3},\curl H_{\k}^{4}) & \eta(\curl H_{\k}^{5},\curl H_{\k}^{6})\end{array}\right]
\]
By definition of $\eta$ and since $\curl\Hi=-\ii q_{\k}\Ei$ we have
$\eta(\curl H_{\k}^{i},\curl H_{\k}^{j})=-q_{\k}^{2}\eta(\Ei,E_{\k}^{j})$,
whence
\[
M_{\k}^{(2)}=-q_{\k}^{2}\left[\begin{array}{ccc}
\eta(E_{\k}^{1},E_{\k}^{2}) & \eta(E_{\k}^{3},E_{\k}^{4}) & \eta(E_{\k}^{5},E_{\k}^{6})\end{array}\right].
\]
Therefore, inverting $M_{\k}^{(2)}$ in \eqref{eq:second method-1}
we obtain 

\[
\nabla q_{\k}=q_{\k}(\gamma(\curl H_{\k}^{1},\curl H_{\k}^{2}),\dots,\gamma(\curl H_{\k}^{5},\curl H_{\k}^{6}))(M_{\k}^{(2)})^{-1}\qquad\text{in }\Omega_{\k},
\]
where
\[
\Omega_{\k}=\{x\in\overline{\Omega}:\bigl|\det\left[\begin{array}{ccc}
\eta(E_{\k}^{1},E_{\k}^{2}) & \eta(E_{\k}^{3},E_{\k}^{4}) & \eta(E_{\k}^{5},E_{\k}^{6})\end{array}\right](x)\bigr|>\frac{\p}{2}\}.
\]
In view of \eqref{eq:second method-0} we have $\overline{\Omega}=\bigcup_{\k\in K}\Omega_{\k}$,
and so $q_{\k}$ can be reconstructed everywhere in $\Omega$ following
the algorithm discussed in the previous subsection, provided that
$q_{\k}$ is known at one point in $\overline{\Omega}$.

\section{\label{sec:The-Maxwell's-Equations}Maxwell's equations}

In this section we prove the results stated in Subsection~\ref{sub:The-Maxwell's-equations}.
Propositions \ref{prop:well-posedness} and \ref{prop:regularity}
are standard results, and their proofs will be detailed for completeness.
On the other hand, Proposition \ref{prop:analyticity global} is less
standard and requires the careful analysis of Maxwell's equations
given in this section.

Let us recall problem \eqref{eq:main maxwell}\begin{subequations}
\label{eq:main maxwell-bis}
\begin{empheq}[left=\empheqlbrace]{align}
&\curl\E=\ii\k\mu\H\qquad\text{in \ensuremath{\Omega},}\label{eq:main maxwell 1}\\ 
&\curl\H=-\ii q_\k\E\qquad\text{in \ensuremath{\Omega},}\label{eq:main maxwell 2}\\
&\E\times\nu=\phi\times\nu\qquad\text{on \ensuremath{\partial\Omega}.} \label{eq:main maxwell 3}
\end{empheq}
\end{subequations}We now justify the introduction of the additional constraint $\H\in\Hmu$
in order to make \eqref{eq:main maxwell-bis} well-posed in the case
$\k=0$. For any $p\in\Hone$ the quantity $(0,\nabla p)$ is a solution
to the homogeneous equation. From \eqref{eq:main maxwell 1} it follows
that $\k\div(\mu\H)=0$ in $\Omega$. Thus, it is natural to impose
in general that
\[
\div(\mu\H)=0\qquad\text{in \ensuremath{\Omega},}
\]
which is meaningful also when $\k=0$. Taking in the above example
$p\in\Hone\setminus\{0\}$ such that $\div(\mu\nabla p)=0$ in $\Omega$,
we have that $(0,\nabla p)$ is still a solution to the homogeneous
equation. This shows that the above condition is not sufficient to
guarantee well-posedness. In this example, we would need a boundary
condition on $p$. In view of \cite[(3.52)]{MONK-2003} we have
\begin{equation}
\curl w\cdot\nu=\div_{\partial\Omega}(w\times\nu),\qquad w\in\Hcurl,\label{eq:monk 3.52}
\end{equation}
whence from \eqref{eq:main maxwell 1} we obtain
\[
\ii\k\mu\H\cdot\nu=\curl\E\cdot\nu=\div_{\bo}(\E\times\nu)=\div_{\bo}(\phi\times\nu)=\curl\phi\cdot\nu\qquad\text{on \ensuremath{\bo}.}
\]
This suggests to assume \eqref{eq:assumption phi} and impose $\mu\H\cdot\nu=0$
on $\bo$. This argument justifies assumption \eqref{eq:assumption phi}
on $\phi$ and the constraint $\H\in\Hmu$.

To simplify our study of \eqref{eq:main maxwell-bis}, we first do
a lifting of the boundary condition $\phi$. Namely, write
\begin{equation}
\E=\EE+\phi,\label{eq:lift}
\end{equation}
where $\EE\in\Hocurl$, and obtain
\begin{equation}
\left\{ \begin{array}{l}
\curl\EE=\ii\k\mu\H-\curl\phi\qquad\text{in \ensuremath{\Omega},}\\
\curl\H=-\ii q_{\k}\EE-\ii q_{\k}\phi\qquad\text{in \ensuremath{\Omega}.}
\end{array}\right.\label{eq:lift maxwell}
\end{equation}

\subsection{Well-posedness}

This subsection is devoted to the proof of Proposition~\ref{prop:well-posedness}.
As far as the case $\k\neq0$ is concerned, we follow \cite{SOMERSALO-ISAACSON-CHENEY-1992}.
The case $\k=0$ is still standard, but requires additional care.

Introduce the space
\[
X=\ld\times\{v\in\ld:\div(\mu v)=0\text{ in }\Omega,\;\mu v\cdot\nu=0\text{ on }\partial\Omega\},
\]
equipped with the norm $\left\Vert (u,v)\right\Vert _{X}^{2}=\left\Vert u\right\Vert _{\ld}^{2}+\left\Vert v\right\Vert _{\ld}^{2}$.
Consider its subspace $\mathcal{D}(T)=\Hocurl\times\Hmu$ and the
operator
\[
T\colon\mathcal{D}(T)\longrightarrow X,\qquad T(u,v)=\ii\left(\epsilon^{-1}(\curl v-\sigma u),-\mu^{-1}\curl u\right)\!.
\]
In view of \eqref{eq:monk 3.52} we have $\curl u\cdot\nu=0$ for
all $u\in\Hocurl$. Therefore $T(u,v)\in X$ for all $(u,v)\in\mathcal{D}(T)$,
and so $T$ is well-defined.

The following lemma states that \eqref{eq:lift maxwell} can be recast
as a Fredholm-type equation involving the operator $T$.
\begin{lem}
\label{lem:equivalence}Assume that \eqref{eq:assumption_ell} and
\eqref{eq:assumption phi} hold and take $\k\in\C$ and $(\EE,\H)\in\mathcal{D}(T)$.
Then $(\EE,\H)$ is solution to \eqref{eq:lift maxwell} if and only
if 
\begin{equation}
(T-\k)(\EE,\H)=(\epsilon^{-1}q_{\k}\phi,\ii\mu^{-1}\curl\phi).\label{eq:T-omega}
\end{equation}
\end{lem}
\begin{proof}
Showing this equivalence is just a matter of writing down the relevant
identities, and the details are left to the reader.
\end{proof}
The previous lemma states the equivalence between our original problem
\eqref{eq:lift maxwell} and the Fredholm-type equation \eqref{eq:T-omega}.
The first natural step towards the study of the latter is the characterization
of the spectrum of $T$, which we will denote by $\Sigma=\sigma(T)$.

The spectrum of an extension $\tilde{T}$ of $T$ was studied in \cite{SOMERSALO-ISAACSON-CHENEY-1992}.
Consider the space $\tilde{X}=\ld\times\ld$ equipped with the norm
$\left\Vert (u,v)\right\Vert _{\tilde{X}}^{2}=\left\Vert u\right\Vert _{\ld}^{2}+\left\Vert v\right\Vert _{\ld}^{2}$,
its subspace $\mathcal{D}(\tilde{T})=\Hocurl\times\Hcurl$ and the
operator
\[
\tilde{T}\colon\mathcal{D}(\tilde{T})\longrightarrow\tilde{X},\qquad\tilde{T}(u,v)=\ii\left(\epsilon^{-1}(\curl v-\sigma u),-\mu^{-1}\curl u\right)\!.
\]

\begin{lem}[{\cite[Proposition 3.1]{SOMERSALO-ISAACSON-CHENEY-1992}}]
\label{lem:spectrum of T tilde}Assume that \eqref{eq:assumption_ell}
holds. The spectrum of $\tilde{T}$ is discrete.
\end{lem}
As it has already been pointed out,  $\tilde{T}(0,\nabla p)=0$
for every $p\in\Hone$, namely $\tilde{T}$ is not injective. Therefore
$0\in\sigma(\tilde{T})$. The restriction we set in this work to the
domain and the codomain of the operator $T$ are motivated by the
need of studying \eqref{eq:lift maxwell}, whence \eqref{eq:T-omega},
also in the case $\k=0$. Thus, we shall now prove that $0\notin\Sigma$.
\begin{lem}
\label{lem:0 notin Sigma}Assume that \eqref{eq:assumption_ell} holds.
The operator $T$ is invertible and $T^{-1}\colon X\to\mathcal{D}(T)$
is continuous, namely
\[
\bigl\Vert(u,v)\bigr\Vert_{X}\le C\left\Vert T(u,v)\right\Vert _{X},\qquad(u,v)\in X,
\]
for some $C>0$ depending on $\Omega$, $\mina$ and $\maxa$ only.\end{lem}
\begin{proof}
Let $(F,G)\in X$. We need to show that there exists a unique $(u,v)\in\mathcal{D}(T)$
such that $T(u,v)=(F,G)$ and that $\bigl\Vert(u,v)\bigr\Vert_{X}\le C\left\Vert (F,G)\right\Vert _{X}$,
for some $C>0$ depending on $\Omega$, $\mina$ and $\maxa$ only.
In the following we shall denote different such constants with the
same letter $c$.

Let us rewrite $ $$T(u,v)=(F,G)$ as
\begin{equation}
\left\{ \begin{array}{l}
\sigma u-\curl v=\ii\epsilon F\qquad\text{in \ensuremath{\Omega},}\\
\curl u=\ii\mu G\qquad\text{in \ensuremath{\Omega}.}
\end{array}\right.\label{eq:T explicit}
\end{equation}
In view of the Helmholtz decomposition \cite[Chapter I, Corollary 3.4]{GIRAULT-RAVIART-1986},
we can write $u=\nabla p+\curl\Phi$ for some $p\in\Hone$ and $\Phi\in H^{1}(\Omega;\C^{3})$
such that $\div\,\Phi=0$ in $\Omega$ and $\Phi\times\nu=0$ on $\partial\Omega$.
Since $\curl(\curl\Phi)=\nabla(\div\Phi)-\Delta\Phi=-\Delta\Phi$,
the second equation of \eqref{eq:T explicit} yields
\[
\left\{ \begin{array}{l}
-\Delta\Phi=\ii\mu G\qquad\text{in \ensuremath{\Omega},}\\
\div\,\Phi=0\qquad\text{in \ensuremath{\Omega},}\\
\Phi\times\nu=0\qquad\text{on \ensuremath{\partial\Omega}.}
\end{array}\right.
\]
Thus $\Phi$ is uniquely determined by $G$ and in view of \cite[Chapter I, Theorem 3.8]{GIRAULT-RAVIART-1986}
there holds
\begin{equation}
\left\Vert \curl\Phi\right\Vert _{H^{1}(\Omega;\C^{3})}\le c\bigl(\left\Vert \curl\curl\Phi\right\Vert _{\ld}+\left\Vert \div\curl\Phi\right\Vert _{L^{2}(\Omega;\C)}\bigr)\le c\left\Vert G\right\Vert _{\ld}.\label{eq:bound on Phi}
\end{equation}

We now want to find suitable boundary conditions satisfied by $p$.
We denote the surface gradient by $\nabla_{\partial\Omega}$, the
surface divergence by $\div_{\bo}$ and the surface scalar curl by
$\curl_{_{\partial\Omega}}$ \cite[Section 3.4]{MONK-2003}. By \eqref{eq:monk 3.52}
and \cite[(3.15)]{MONK-2003} we have
\[
0=\ii\mu G\cdot\nu=\curl\curl\Phi\cdot\nu=\div_{\partial\Omega}(\curl\Phi\times\nu)=\curl_{_{\partial\Omega}}\curl\Phi\qquad\text{on \ensuremath{\partial\Omega}}.
\]
(Note that $\curl\Phi\cdot\nu=\div_{\partial\Omega}(\Phi\times\nu)=0$
on $\partial\Omega$, so that $\curl\Phi$ is a tangential vector
field, and so we can apply to it the surface scalar curl.) As a result,
since $\bo$ is simply connected, there exists a unique $r\in H^{1}(\bo;\C)$
such that $\curl\Phi=-\nabla_{\bo}r$ on $\bo$ and $\int_{\bo}r\, ds=0$.
Poincaré inequality gives
\begin{equation}
\left\Vert r\right\Vert _{H^{1}(\bo;\C)}\le c\left\Vert \nabla_{\bo}r\right\Vert _{L_{t}^{2}(\bo,\C^{3})}=c\left\Vert \curl\Phi\right\Vert _{L_{t}^{2}(\bo,\C^{3})},\label{eq:poincare}
\end{equation}
where $L_{t}^{2}(\bo,\C^{3})$ denotes the space of tangential vector
fields in $L^{2}(\bo,\C^{3})$. As $u\times\nu=0$ on $\bo$ we have
$\nabla_{\bo}r\times\nu=-\curl\Phi\times\nu=\nabla p\times\nu=\nabla_{\bo}p\times\nu$
on $\bo$, whence $\nabla_{\bo}p=\nabla_{\bo}r$ on $\bo$. Since
$p$ is defined up to a constant, we can set $p=r$ on $\bo$. Thus,
in view of the first equation of \eqref{eq:T explicit}, we must look
for a solution to
\[
\left\{ \begin{array}{l}
-\div(\sigma\nabla p)=\div(\sigma\curl\Phi-\ii\epsilon F)\qquad\text{in \ensuremath{\Omega},}\\
p=r\qquad\text{on \ensuremath{\partial\Omega}.}
\end{array}\right.
\]
Therefore, $p$ is uniquely determined by $\Phi$, $F$ and $r$ and
the following estimate holds
\begin{equation}
\begin{split}\left\Vert p\right\Vert _{\Hone} & \le c\left(\left\Vert \curl\Phi\right\Vert _{\ld}+\left\Vert F\right\Vert _{\ld}+\left\Vert r\right\Vert _{H^{1}(\bo;\C)}\right)\\
 & \le c\left(\left\Vert \curl\Phi\right\Vert _{H^{1}(\Omega;\C^{3})}+\left\Vert F\right\Vert _{\ld}\right)\\
 & \le c\left(\left\Vert G\right\Vert _{\ld}+\left\Vert F\right\Vert _{\ld}\right),
\end{split}
\label{eq:bound on p}
\end{equation}
where the second inequality is a consequence of \eqref{eq:poincare}
and the third one follows from \eqref{eq:bound on Phi}. We have proven
that $u$ is uniquely determined by $F$ and $G$ and combining \eqref{eq:bound on Phi}
and \eqref{eq:bound on p} gives the estimate
\[
\bigl\Vert u\bigr\Vert_{\ld}\le c\left\Vert (F,G)\right\Vert _{X}.
\]

The well-posedness for $v$ follows from \cite[Chapter I, Theorem 3.5]{GIRAULT-RAVIART-1986}
applied to the first equation of \eqref{eq:T explicit}. Indeed, $\div(\sigma u-\ii\epsilon F)=0$
in $\Omega$ and so there exists a unique $v\in\Hcurl$ such that
\[
\left\{ \begin{array}{l}
\curl v=\sigma u-\ii\epsilon F\qquad\text{in \ensuremath{\Omega},}\\
\div(\mu v)=0\qquad\text{in \ensuremath{\Omega},}\\
\mu v\cdot\nu=0\qquad\text{on \ensuremath{\partial\Omega},}
\end{array}\right.
\]
and the norm estimate $\left\Vert v\right\Vert _{\ld}\le c(\bigl\Vert u\bigr\Vert_{\ld}+\bigl\Vert F\bigr\Vert_{\ld})$
holds.

Finally, we have proven that there exists a unique $(u,v)\in\mathcal{D}(T)$
such that \eqref{eq:T explicit} holds true. Moreover, thanks to the
estimates on the norms of $u$ and $v$, $T^{-1}$ is continuous.
\end{proof}
Combining Lemmata~\ref{lem:spectrum of T tilde} and \ref{lem:0 notin Sigma}
and using the fact that $\Sigma\subseteq\sigma(\tilde{T})$ we obtain
the following characterization of the spectrum of $T$.
\begin{prop}
\label{prop:spectrum of T}Assume that \eqref{eq:assumption_ell}
holds. The spectrum $\Sigma$ of $T$ is discrete. Moreover, $0\notin\Sigma$.
In particular, for all $w\in\C\setminus\Sigma$ and $(F,G)\in X$
the equation
\begin{equation}
(T-\k)(u,v)=(F,G)\label{eq:T - omega F,G}
\end{equation}
has a unique solution $(u,v)\in\mathcal{D}(T)$ such that
\[
\bigl\Vert(u,v)\bigr\Vert_{X}\le C\left\Vert (F,G)\right\Vert _{X},
\]
for some $C>0$ independent of $F$ and $G$.
\end{prop}
We are now in a position to prove Proposition~\ref{prop:well-posedness}.
\begin{proof}[Proof of Proposition~\ref{prop:well-posedness}]
For $\k\in\C\setminus\Sigma$, set $F=\epsilon^{-1}q_{\k}\phi$ and
$G=\ii\mu^{-1}\curl\phi$ in \eqref{eq:T - omega F,G}. In view of
Proposition~\ref{prop:spectrum of T},  equation \eqref{eq:T-omega}
has a unique solution $(\EE,\H)\in\mathcal{D}(T)$ satisfying the
norm estimate
\[
\bigl\Vert(\EE,\H)\bigr\Vert_{X}\le C\left\Vert \phi\right\Vert _{\Hcurl},
\]
for some $C$ independent of $\phi$. After setting $\E=\EE+\phi$
as in \eqref{eq:lift}, Lemma~\ref{lem:equivalence} implies that
$(\E,\H)$ is the unique solution to \eqref{eq:combined}. Moreover,
the relevant norm estimate holds. This completes the proof.
\end{proof}

\subsection{Regularity properties}

The focus of this subsection is the proof of Proposition~\ref{prop:regularity}.
First, we study the regularity of the solutions to \eqref{eq:T - omega F,G}
by applying the regularity theory for Maxwell's equations discussed
in \cite{ALBERTI-CAPDEBOSCQ-2013}.
\begin{prop}
\label{pro:regularity lemma}Assume that \eqref{eq:assumption_ell}
and \eqref{eq:regularity} hold. Take $\k\in\C\setminus\Sigma$ and
$(F,G)\in X$ such that $F\in W^{\order,p}(\Omega;\C^{3})$ and $G\in W^{\order-1,p}(\Omega;\C^{3})$.
Let $(u,v)\in\mathcal{D}(T)$ be a solution to \eqref{eq:T - omega F,G}.
Then $u,v\in W^{\order,p}(\Omega;\C^{3})$ and
\begin{equation}
\left\Vert (u,v)\right\Vert _{W^{\order,p}(\Omega;\C^{3})^{2}}\le C\left(\left\Vert (u,v)\right\Vert _{\ld^{2}}+\left\Vert F\right\Vert _{W^{\order,p}(\Omega;\C^{3})}+\left\Vert G\right\Vert _{W^{\order-1,p}(\Omega;\C^{3})}\right),\label{eq:estimate lemma regularity}
\end{equation}
for some $C>0$ depending on $\Omega$, $\k$, $\mina$, $\maxa$,
$\kappa$, $p$ and $\left\Vert (\mu,\epsilon,\sigma)\right\Vert _{W^{\order,p}(\Omega;\R^{3\times3})^{3}}$
only. Moreover, if $\k=0$ we have
\[
\left\Vert (u,v)\right\Vert _{W^{\order,p}(\Omega;\C^{3})^{2}}\le C'\left(\left\Vert F\right\Vert _{W^{\order,p}(\Omega;\C^{3})}+\left\Vert G\right\Vert _{W^{\order-1,p}(\Omega;\C^{3})}\right),
\]
for some $C'>0$ depending on $\Omega$, $\mina$, $\maxa$, $\kappa$,
$p$ and $\left\Vert (\mu,\epsilon,\sigma)\right\Vert _{W^{\order,p}(\Omega;\R^{3\times3})^{3}}$
only.\end{prop}
\begin{proof}
By definition of $T$, \eqref{eq:T - omega F,G} can be rewritten
as
\[
\left\{ \begin{array}{l}
-q_{\k}u+\ii\curl v=\epsilon F\qquad\text{in \ensuremath{\Omega},}\\
-\ii\mu^{-1}\curl u-\k v=G\qquad\text{in \ensuremath{\Omega}.}
\end{array}\right.
\]
Substituting $v\to-v$, this problem can be recast as
\[
\left\{ \begin{array}{l}
\curl v=\ii\k'\epsilon'u+J_{e}\qquad\text{in \ensuremath{\Omega},}\\
\curl u=-\ii\k'\mu'v+J_{m}\qquad\text{in \ensuremath{\Omega},}
\end{array}\right.
\]
where $\k'=1$, $\epsilon'=\k\epsilon+\ii\sigma$, $J_{e}=\ii\epsilon F$,
$\mu'=\k\mu$ and $J_{m}=\ii\mu G$. This system's form was considered
in \cite{ALBERTI-CAPDEBOSCQ-2013}. In view of \cite[Proposition 5]{ALBERTI-CAPDEBOSCQ-2013}
we obtain that $(u_{l},v_{l})$ satisfies for any $l=1,2,3$ and all
$g\in H^{2}(\Omega;\C)$
\begin{multline*}
 \int_{\Omega}u_{l}\div(^{t}\!\epsilon'\nabla\bar{g})\, dx=\int_{\bo}(\partial_{l}\bar{g})\epsilon'u\cdot\nu\, ds\\+\int_{\Omega}\left(\partial_{l}\epsilon'u-\epsilon'(\e_{l}\times(\ii\mu G-\ii\k\mu v))-\ii\e_{l}\div(\ii\epsilon F)\right)\cdot\nabla\bar{g}\, dx,
\end{multline*}

and
\begin{multline*}
\int_{\Omega}v_{l}\div(^{t}\!(\k\mu)\nabla\bar{g})\, dx=-\int_{\bo}(\e_{l}\times(v\times\nu))\cdot(^{t}(\k\mu)\nabla\bar{g})\, ds\\
+\int_{\Omega}\left(\k(\partial_{l}\mu)v-\k\mu(\e_{l}\times(i\epsilon F+\ii(\k\epsilon+\ii\sigma)u))\right)\cdot\nabla\bar{g}\, dx.
\end{multline*}
Multiplying the first equation by $-\ii$ we obtain
\begin{multline}
\int_{\Omega}u_{l}\div(^{t}\!(\sigma-\ii\k\epsilon)\nabla\bar{g})\, dx=\int_{\bo}(\partial_{l}\bar{g})(\sigma-\ii\k\epsilon)u\cdot\nu\, ds\\
+\int_{\Omega}\left(\partial_{l}(\sigma-\ii\k\epsilon)u-(\sigma-\ii\k\epsilon)(\e_{l}\times(\ii\mu G-\ii\k\mu v))-\ii\e_{l}\div(\epsilon F)\right)\cdot\nabla\bar{g}\, dx.\label{eq:weak 1}
\end{multline}
Similarly, we can multiply the second equation by $\k^{-1}$ and obtain
\begin{multline}
\int_{\Omega}v_{l}\div(^{t}\!\mu\nabla\bar{g})\, dx=-\int_{\bo}(\e_{l}\times(v\times\nu))\cdot(^{t}\mu\nabla\bar{g})\, ds\\
+\int_{\Omega}\left((\partial_{l}\mu)v-\mu(\e_{l}\times(i\epsilon F+\ii(\k\epsilon+\ii\sigma)u))\right)\cdot\nabla\bar{g}\, dx.\label{eq:weak 2}
\end{multline}
Note that this last operation is allowed if $\k\neq0$. If $\k=0$,
we can argue as follows. Take a sequence $\k_{n}\in\C\setminus(\Sigma\cup\{0\})$
such that $\k_{n}\to0$. Since the resolvent operator $(T-\k)^{-1}$
is analytic in $\k$, then the map $\k\in\C\setminus\Sigma\mapsto(T-\k)^{-1}(F,G)\in X$
is continuous. Therefore, as \eqref{eq:weak 2} is satisfied for all
$\k_{n}$, it has to hold also for $\k=0$.

A careful analysis shows that the proof of \cite[Theorem 9]{ALBERTI-CAPDEBOSCQ-2013}
only relies on the system \eqref{eq:weak 1}-\eqref{eq:weak 2} and
not on the original form of Maxwell's equations considered in \cite{ALBERTI-CAPDEBOSCQ-2013} (see also \cite{2014albertidphil}).
Therefore, it is possible to apply this result to the previous system
\eqref{eq:weak 1}-\eqref{eq:weak 2} and obtain that $u,v\in W^{\order,p}(\Omega;\C^{3})$
and
\[
\left\Vert (u,v)\right\Vert _{W^{\order,p}(\Omega;\C^{3})^{2}}\le c\left(\left\Vert (u,v)\right\Vert _{\ld^{2}}+\left\Vert \epsilon F\right\Vert _{W^{\order,p}(\div,\Omega)}+\left\Vert \mu G\right\Vert _{W^{\kappa,p}(\div,\Omega)}\right),
\]
for some $c>0$ depending on $\Omega$, $\k$, $\mina$, $\maxa$,
$\kappa$, $p$ and $\left\Vert (\mu,\epsilon,\sigma)\right\Vert _{W^{\order,p}(\Omega;\R^{3\times3})^{3}}$
only, where
\[
W^{\order,p}(\div,\Omega)=\{u\in W^{\kappa-1,p}(\Omega;\C^{3}):\div u\in W^{\kappa-1,p}(\Omega;\C)\}.
\]
Estimate \eqref{eq:estimate lemma regularity} follows from $\left\Vert \epsilon F\right\Vert _{W^{\order,p}(\Omega;\C^{3})}\le c\left\Vert F\right\Vert _{W^{\order,p}(\Omega;\C^{3})}$,
which is an easy consequence of the Sobolev Embedding Theorem.

Finally, the last estimate follows from \eqref{eq:estimate lemma regularity}
and Lemma~\ref{lem:0 notin Sigma}.
\end{proof}
We are now ready to prove Proposition~\ref{prop:regularity}.
\begin{proof}[Proof of Proposition~\ref{prop:regularity}]
We have already seen (see the proof of Proposition~\ref{prop:well-posedness})
that setting $\E=\EE+\phi$, $(\EE,\H)$ is a solution to \eqref{eq:T - omega F,G}
with $F=\epsilon^{-1}q_{\k}\phi$ and $G=\ii\mu^{-1}\curl\phi$. Thus,
in view of Proposition~\ref{pro:regularity lemma}, $\EE,\H\in W^{\order,p}(\Omega;\C^{3})$
and
\begin{multline*}
\!\!\!\!\bigl\Vert(\EE,\H)\bigr\Vert_{W^{\order,p}(\Omega;\C^{3})^{2}}\le C\Bigl(\bigl\Vert(\EE,\H)\bigr\Vert_{\ld^{2}}\\+\left\Vert \epsilon^{-1}q_{\k}\phi\right\Vert _{W^{\order,p}(\Omega;\C^{3})}+\left\Vert \mu^{-1}\curl\phi\right\Vert _{W^{\order-1,p}(\Omega;\C^{3})}\Bigr)\!,\!\!\!\!
\end{multline*}
for some $C>0$ depending on $\Omega$, $\k$, $\mina$, $\maxa$,
$\kappa$, $p$ and $\left\Vert (\mu,\epsilon,\sigma)\right\Vert _{W^{\order,p}(\Omega;\R^{3\times3})^{3}}$
only. Combining this inequality with the estimate for $\bigl\Vert(\E,\H)\bigr\Vert_{\Hcurl^{2}}$
given in Proposition~\ref{prop:well-posedness} we obtain the first
part of the result, as $W^{\kappa,p}(\Omega;\C^{3})^{2}$ is continuously
embedded into $\Cl^{\kappa-1}(\overline{\Omega};\C^{6})$. Finally,
the estimate on $\left\Vert (E_{0},H_{0})\right\Vert _{W^{\order,p}(\Omega;\C^{3})^{2}}$
follows from the second part of Proposition~\ref{pro:regularity lemma}.
\end{proof}

\subsection{Analyticity properties}

We end this section by proving Proposition~\ref{prop:analyticity global}.
This will be a consequence of the results discussed so far in this
section and of the following lemma, which generalizes \cite[Proposition~3.5]{alberti2013multiple}.
\begin{lem}
\label{lem:general analyticity}Let $Y$ be a Banach space. Take an
operator $T\colon\mathcal{D}(T)\subseteq Y\to Y$ and denote its resolvent
set by $\rho(T)$. Take $Y_{1}\subseteq Y$ and $Y_{2}\subseteq\mathcal{D}(T)\cap Y_{1}$
with norms $\left\Vert \;\right\Vert _{Y_{1}}$ and $\left\Vert \;\right\Vert _{Y_{2}}$,
respectively. Assume that the inclusion $i\colon Y_{2}\to Y_{1}$
is continuous and that for all $\k\in\rho(T)$ the operator $(T-\k)^{-1}\colon Y_{1}\to Y_{2}$
is well-defined and bounded. Take $N\in Y_{2}$ and let $g\colon\rho(T)\to Y_{1}$
be such that $g(\k)-g(\k_{0})=(\k-\k_{0})N$ for all $\k,\k_{o}\in\rho(T)$.
Then the map
\[
\k\in\rho(T)\longmapsto(T-\k)^{-1}g(\k)\in Y_{2}
\]
is analytic.\end{lem}
\begin{proof}
Denote the map $\k\in\rho(T)\longmapsto(T-\k)^{-1}g(\k)\in Y_{2}$
by $f$. Take $\k_{0}\in\rho(T)$: we shall prove that $f$ is analytic
in $\k_{0}$. Denote the operator $(T-\k_{0})^{-1}\colon Y_{1}\to Y_{2}$
by $B$. Let $r>0$ be such that $B(\k_{0},r)\subseteq\rho(T)$, take
$\k\in B(\k_{0},r)$ and set $h=\k-\k_{0}$. Introduce the operator
$C_{h}\colon Y_{2}\to Y_{1}$ defined by $y\mapsto hy$, whose norm
satisfies $\left\Vert C_{h}\right\Vert \le\left\Vert i\right\Vert r$.
A straightforward calculation shows that
\[
(T-\k_{0})(f(\k)-f(\k_{0}))-h(f(\k)-f(\k_{0}))=h(N+f(\k_{0})),
\]
where the equality makes sense in $Y_{1}$. We can now apply the operator
$B$ to both sides of this equation and get
\[
(I-BC_{h})(f(\k)-f(\k_{0}))=BC_{h}(N+f(\k_{0})).
\]
If  $r<(\left\Vert i\right\Vert \left\Vert B\right\Vert )^{-1}$ we
have that $\left\Vert BC_{h}\right\Vert <1$ and so
\[
f(\k)-f(\k_{0})=\sum_{n=1}^{\infty}(BC_{h})^{n}(N+f(\k_{0})),\qquad\k\in B(\k_{0},r),
\]
whence the result.
\end{proof}
We are now in a position to prove Proposition~\ref{prop:analyticity global}.
\begin{proof}[Proof of Proposition~\ref{prop:analyticity global}]
Since $\E=\EE+\phi$, it is enough to show the analyticity of the
map $\k\mapsto(\EE,\H)$. We want to apply Lemma~\ref{lem:general analyticity}
with $Y=X$, $Y_{1}=\{(F,G)\in X:F\in W^{\order,p}(\Omega;\C^{3}),\; G\in W^{\order-1,p}(\Omega;\C^{3})\}$
equipped with the norm $\left\Vert \;\right\Vert _{W^{\order,p}(\Omega;\C^{3})}\times\left\Vert \;\right\Vert _{W^{\order-1,p}(\Omega;\C^{3})}$,
$Y_{2}=\mathcal{D}(T)\cap W^{\order,p}(\Omega;\C^{3})^{2}$ equipped
with the norm $\left\Vert \;\right\Vert _{W^{\order,p}(\Omega;\C^{3})^{2}}$,
$N=(0,\phi)$ and $g(\k)=(\epsilon^{-1}q_{\k}\phi,\ii\mu^{-1}\curl\phi)$.
Let us now check that the assumptions of the lemma are verified. The
continuity of the inclusion $Y_{2}\subseteq Y_{1}$ is trivial. The
continuity of $(T-\k)^{-1}\colon Y_{1}\to Y_{2}$ follows from Propositions
\ref{prop:spectrum of T} and \ref{pro:regularity lemma}. Finally,
a direct computation shows that $g(\k)-g(\k_{0})=(\k-\k_{0})N$. Therefore
the result follows by Lemma~\ref{lem:general analyticity}, as $W^{\kappa,p}(\Omega;\C^{3})^{2}$
is continuously embedded into $\Cl^{\kappa-1}(\overline{\Omega};\C^{6})$.
\end{proof}

\section*{Acknowledgments}

This work has been done during my D.Phil.~at the Oxford Centre for
Nonlinear PDE under the supervision of Yves Capdeboscq, whom I would
like to thank for many stimulating discussions on the topic and for
a thorough revision of the paper. I have benefited from the support of the EPSRC Science \& Innovation Award to the Oxford
Centre for Nonlinear PDE (EP/EO35027/1).

\bibliographystyle{abbrvurl}
\bibliography{2013-gsalberti}

\end{document}